\documentclass[11pt]{amsart}
\usepackage[margin=1.35in]{geometry}
\usepackage{amscd,amsmath,amsxtra,amssymb,stmaryrd,xr,mathrsfs,mathtools,enumerate,commath, comment,appendix}
\usepackage{xcolor}
\usepackage{comment}
\usepackage{tikz-cd}
\usepackage{float}
\usepackage{longtable} % for 'longtable' environment
\usepackage{pdflscape} % for 'landscape' environment
\usepackage{amsthm}
\usepackage{thmtools}

\usepackage[normalem]{ulem}

\usetikzlibrary{shapes.geometric}
\usetikzlibrary{decorations.markings}
\tikzset{every loop/.style={min distance=10mm,looseness=10}}

\usepackage{multirow}

\newcommand{\DV}[1]{\textcolor{blue}{#1}}
\newcommand{\RG}[1]{\textcolor{red}{#1}}
\newcommand{\KM}[1]{\textcolor{purple}{#1}}
\definecolor{Green}{rgb}{0.0, 0.5, 0.0}
\newcommand{\green}[1]{\textcolor{Green}{#1}}

\newcommand{\cBF}{\mathcal{BF}}

\newcommand{\ZZ}{\mathbb{Z}}

\newcommand{\CC}{\mathbb{C}}
\usepackage[utf8]{inputenc}
\DeclareFontFamily{U}{wncy}{}
\DeclareFontShape{U}{wncy}{m}{n}{<->wncyr10}{}
\DeclareSymbolFont{mcy}{U}{wncy}{m}{n}
\DeclareMathSymbol{\Sh}{\mathord}{mcy}{"58}
\usepackage[T2A,T1]{fontenc}
\usepackage[OT2,T1]{fontenc}

\newtheorem{theorem}{Theorem}[section]
\newtheorem{lemma}[theorem]{Lemma}

\newtheorem{proposition}[theorem]{Proposition}
\newtheorem{corollary}[theorem]{Corollary}
\newtheorem{definition}[theorem]{Definition}

\numberwithin{equation}{section}

\theoremstyle{remark}
\newtheorem{remark}[theorem]{Remark}

\newcommand{\BF}{\mathrm{BF}}
\newcommand{\coker}{\textup{coker}}
\newcommand{\Z}{\mathbb{Z}}
\renewcommand{\Im}{\mathrm{Im}}

\usepackage[hypertexnames=false]{hyperref}
\usepackage[capitalize,nameinlink,noabbrev]{cleveref}

\theoremstyle{plain}
\newtheorem{lthm}{Theorem} % theorems with letters (for intro)

\definecolor{vegasgold}{rgb}{0.77, 0.7, 0.35}
\definecolor{darkgoldenrod}{rgb}{0.72, 0.53, 0.04}
\definecolor{gold(metallic)}{rgb}{0.83, 0.69, 0.22}
\hypersetup{
 colorlinks=true,
 linkcolor=darkgoldenrod,
 filecolor=brown,      
 urlcolor=gold(metallic),
 citecolor=darkgoldenrod,
 pdftitle={Bowwen--Franks groups and minus class groups},
 }

\AddToHook{cmd/appendix/before}{\crefalias{section}{appendix}}

\begin{document}
\title[Bowen--Franks groups and class groups]{Bowen--Franks groups and minus class groups of cyclotomic number fields with prime conductor}

\author[A. Lei]{Antonio Lei}
\address[Lei]{Department of Mathematics and Statistics\\University of Ottawa\\
150 Louis-Pasteur Pvt\\
Ottawa, ON\\
Canada K1N 6N5}
\email{antonio.lei@uottawa.ca}

\author[K. Müller]{Katharina Müller}
\address[Müller]{Institut für Theoretische Informatik, Mathematik und Operations Research, Universität der Bundeswehr München, Werner-Heisenberg-Weg 39, 85577 Neubiberg, Germany}
\email{katharina.mueller@unibw.de}

\author[D.~Valli\`{e}res]{Daniel Valli\`{e}res}
\address[Valli\`{e}res]{Mathematics and Statistics Department, California State University, Chico, CA 95929, USA}
\email{dvallieres@csuchico.edu}

\begin{abstract}
Let $p$ be an odd rational prime and consider the cyclotomic number field $K = \mathbb{Q}(\zeta_{p})$ of conductor $p$.  We construct a directed graph $Y$ on $p-1$ vertices for which the torsion part of the corresponding Bowen--Franks group is closely related to the minus part of the class group of $K$. In particular, both groups have the same cardinality up to an explicit power of $p$. Furthermore, they are both $\mathrm{Gal}(K/\mathbb{Q})$-modules, and we prove the equality of the cardinalities of their isotypic components after tensoring them with the valuation ring of an appropriate $\ell$-adic field for $\ell \nmid p-1$.
\end{abstract}

\subjclass[2020]{Primary: 05C20, 11R29; Secondary: 37B10} 
\date{\today} 
\keywords{Ideal class groups of cyclotomic fields, covering of digraphs, symbolic dynamics}

\maketitle
 
\tableofcontents

\section{Introduction}
Let $p$ be a rational prime, and consider the cyclotomic number field $K = \mathbb{Q}(\zeta_{p})$, where $\zeta_{p}$ is a primitive $p$-th root of unity in an algebraic closure $\overline{\mathbb{Q}}$ of $\mathbb{Q}$.  We let $\Delta = {\rm Gal}(K/\mathbb{Q}) \simeq \mathbb{F}_{p}^{\times}$ and set $S =\{\infty,p \}$.  A central object in the study of the ideal class group of $K$ is the special value at $s=0$ of the equivariant $L$-function $\theta_{K/\mathbb{Q},S}(s)$ associated with the Galois extension $K/\mathbb{Q}$.  By classical results, it is given by 
$$\theta_{K/\mathbb{Q},S}(0) = \sum_{i=1}^{p-1} \left(\frac{1}{2} - \frac{i}{p} \right) \sigma_{i}^{-1} \in \mathbb{Q}[\Delta],$$
where $\sigma_{i} \in \Delta$ is the usual Galois automorphism defined via $\zeta_{p} \mapsto \sigma_{i}(\zeta_{p}) = \zeta_{p}^{i}$.  Given a unital commutative ring $R$ and a finite abelian group $G$, we let $\widehat{G}(R) = {\rm Hom}_{\mathbb{Z}}(G,R^{\times})$.  If $\psi \in \widehat{\Delta}(\mathbb{C})$ is a complex valued character of $\Delta$, then it induces a $\mathbb{Q}$-algebra morphism $\mathbb{Q}[\Delta] \rightarrow \mathbb{C}$, which we denote by the same symbol $\psi$.  For all $\psi \in \widehat{\Delta}(\mathbb{C})$ different from the trivial character $\psi_{0}$, one has
\begin{equation} \label{dirichlet}
\overline{\psi}(\theta_{K/\mathbb{Q},S}(0)) = L(0,\psi),
\end{equation}
where $L(s,\psi)$ is the corresponding primitive Dirichlet $L$-function, and the bar denotes complex conjugation.  From now on, consider the element
\begin{equation} \label{stickelberger}
\theta = - \frac{1}{p}\sum_{i=1}^{p-1} i \sigma_{i}^{-1} \in \mathbb{Q}[\Delta]. 
\end{equation}
Then, for all non-trivial character $\psi \in \widehat{\Delta}(\mathbb{C})$, one has
\begin{equation} \label{stick_vs_sv}
\psi(\theta) = \psi(\theta_{K/\mathbb{Q},S}(0)), 
\end{equation}
from which it follows that, away from the trivial character, $\theta$ contains the same information as $\theta_{K/\mathbb{Q},S}(0)$.  The element $-\theta$ is often called the Stickelberger element in the literature.  Moreover, if $F$ is any field of characteristic zero containing the $(p-1)$-th roots of unity, and $\psi \in \widehat{\Delta}(F)$ is a non-trivial character, then
\begin{equation} \label{bernoulli_def}
B_{1,\psi} = -\psi^{-1}(\theta) = \frac{1}{p} \sum_{i=1}^{p-1}i\psi(\sigma_{i}) \in F
\end{equation}
is the usual generalized Bernoulli $F$-number corresponding to the character $\psi \in \widehat{\Delta}(F)$.  Putting (\ref{dirichlet}), (\ref{stick_vs_sv}) and (\ref{bernoulli_def}) together gives
$$L(0,\psi) = -B_{1,\psi},$$
for any non-trivial character $\psi \in \widehat{\Delta}(\mathbb{C})$.
  
Now, we have
\begin{equation} \label{integralized_stick}
p \theta = - \sum_{i=1}^{p-1}i \sigma_{i}^{-1} \in \mathbb{Z}[\Delta],
\end{equation}
and another way to construct elements in the group ring $\mathbb{Z}[\Delta]$ is via the special value at $u=1$ of equivariant Ihara zeta functions in graph theory.  The Ihara zeta function of a finite graph was introduced by Ihara in \cite{Ihara:1966} in his study of cocompact discrete torsion-free subgroups $\Gamma$ of ${\rm PGL}(2,\mathbb{Q}_{p})$ acting on a Bruhat-Tits tree $Y$.  Serre suggested in \cite{Serre:1977} that this zeta function could be reformulated in terms of the finite graph $X = \Gamma \backslash Y$ only, and this was explicitly done by Sunada in \cite{Sunada:1986}.  Many authors incorporated representations of the fundamental group $\pi_{1}(X,v_{0})$ into the picture in order to study Ihara zeta functions twisted by representations.  The resulting functions are analogous to Artin $L$-functions in algebraic number theory. See, for instance, the work of Hashimoto in \cite{Hashimoto:1989, Hashimoto:1990, Hashimoto:1992} and also the works of Stark and Terras in \cite{Stark/Terras:1996, Stark/Terras:2000, Stark/Terras:2007} and \cite{Terras:2011}. Let us assume for the moment that there exists a Galois cover $Y/X$ of undirected graphs with Galois group $\Delta$ for which 
$$\eta_{Y/X}(1) = p \theta, $$
where $\eta_{Y/X}(u) \in \mathbb{Z}[\Delta][u]$ is the polynomial defined in \cite[(3.4)]{Vallieres/Wilson:2025}.  The polynomial $\eta_{Y/X}(u)$ is closely related to the equivariant Ihara zeta function of the cover $Y/X$; the precise relationship is given by \cite[(3.5)]{Vallieres/Wilson:2025}.  Then, combining \cite[Corollary 4.3]{Vallieres/Wilson:2025} with Mazur and Wiles's refinement of the Herbrand--Ribet theorem (\cite[Theorem 2, page 216]{Mazur/Wiles:1984}, see also \cite[page 301]{Washington:1997}) would give the equality
$$\#e_{\omega^{i}} {\rm Pic}^{0}_{p}(Y) = p|B_{1,\omega^{-i}}|_{p}^{-1} = p \cdot \# e_{\omega^{i}}{\rm Cl}^{-}_{p}(K), $$
for all odd integers $i$ satisfying $3 \le i  \le p-1$, where $\omega:\Delta \simeq \mathbb{F}_{p}^{\times} \hookrightarrow \mathbb{Z}_{p}^{\times}$ is the Teichm\"{u}ller character, $e_{\omega^{i}} \in \mathbb{Z}_{p}[\Delta]$ is the idempotent corresponding to the character $\omega^{i}$, the numbers $B_{1,\omega^{-i}}$ are the usual generalized Bernoulli numbers defined above in (\ref{bernoulli_def}), ${\rm Pic}^{0}_{p}(Y)$ is the Sylow $p$-subgroup of the Picard group of degree zero of $Y$, and ${\rm Cl}^{-}_{p}(K)$ is the Sylow $p$-subgroup of the minus class group of the cyclotomic number field $K$.  Currently, we do not know if there is a way to realize $p \theta$ as the special value at $u=1$ of an equivariant Ihara zeta function in the category of undirected graphs, and this seems to be due to the extra structure imposed by the inversion map appearing in Serre's formalism. We note that many questions about the cardinality of ${\rm Cl}^{-}(K)$ remain open.  For example, the asymptotic behavior of these numbers, encapsulated in Kummer's conjecture, remains open; see, for instance, \cite{granville,murtypetridis}. 

In this article, we will instead work with directed graphs to establish such a relation.  To make clear the distinction between an undirected graph and a directed graph, we will simply call the former a graph, and the latter a digraph.  Our point of view is that a graph is a digraph with an extra structure given by the inversion map in Serre's formalism. In this article, our emphasis will be on digraphs only.  The Picard group of an undirected graph is replaced with the Bowen--Franks group of a digraph arising in symbolic dynamics, and the Ihara zeta function is replaced with another zeta function which is a special instance of an Artin--Mazur zeta function; it is the Artin--Mazur zeta function associated with a shift of finite type, namely the edge shift corresponding to the digraph.  The Bowen--Franks group was introduced in \cite{Bowen/Franks:1977}, and the Artin--Mazur zeta function was introduced in \cite{Artin/Mazur:1965}, originally in the context of differential geometry.  Bowen and Lanford showed in \cite{Bowen/Lanford:1970} that the zeta function associated with a strongly connected finite digraph is a rational function of a particularly simple shape, since it is the inverse of the reverse characteristic polynomial of the adjacency matrix of the digraph (see also \cite{Kotani/Sunada:2000}).  For that reason, the zeta function associated with a digraph is sometimes called a Bowen--Lanford zeta function in the literature.  Now, in the category of digraphs, it is possible to construct a Galois cover $Y/X$ for which the special value at $u=1$ of the equivariant zeta function is precisely equal to $p\theta$ as defined above in (\ref{integralized_stick}) (see \cref{a_special_cover}).  It seems plausible to expect that the special value at $u=1$ of the zeta function of a digraph $Y$ contains information about the Bowen--Franks group of $Y$, and also that an analogue of the Herbrand-Ribet theorem, as well as its refinement by Mazur and Wiles, hold true for digraphs as well.  The same line of reasoning from the previous paragraph suggests then a potential connection between some Bowen--Franks groups in symbolic dynamics and some class groups in algebraic number theory.  The goal of this article is to explore such a potential connection for the cyclotomic number fields $K$ of prime conductor.  The first main result of this article is the following.
\begin{lthm}[\cref{goal}] \label{main1}
Let $p$ be an odd rational prime.  There exists a strongly connected finite digraph $Y$ with $p-1$ vertices for which
$$\#{\rm BF}(Y)_{\rm tors} = p^{\frac{p-1}{2}} \cdot \#{\rm Cl}^{-}(K),$$
where ${\rm BF}(Y)$ denotes the Bowen--Franks group of $Y$ and ${\rm Cl}^{-}(K)$ the minus class group of the cyclotomic number field $K = \mathbb{Q}(\zeta_{p})$.
\end{lthm}
The digraph $Y$ appearing in \cref{main1} is a Galois cover of a bouquet digraph $X$ with Galois group $\Delta$; it follows that ${\rm BF}(Y)$ is a $\mathbb{Z}[\Delta]$-module.  \cref{main1} is closely related to a classical theorem of Iwasawa (\cite[Theorem 1]{Iwasawa:1961}) reinterpreting the analytic class number formula on the minus side in terms of the index of the Stickelberger ideal $I^{-}$ inside the minus part of $\mathbb{Z}[\Delta]$, the formal statement being the equality
$$[\mathbb{Z}[\Delta]^{-}:I^{-}] = \#{\rm Cl}^{-}(K).$$
\cref{main1} is also closely related to the work of Schoof in \cite{Schoof:1998}.  We let $\ell$ be a rational prime satisfying 
\begin{equation} \label{div_for_intro}
\ell \nmid \# \Delta = p-1,
\end{equation}
and set $\mathcal{O}$ to be the valuation ring of the local field $\mathbb{Q}_{\ell}(\mu_{p-1})$, where $\mu_{p-1}$ denotes the collection of $(p-1)$-th roots of unity in an algebraic closure $\overline{\mathbb{Q}}_{\ell}$ of $\mathbb{Q}_{\ell}$.  Given $\psi \in \widehat{\Delta}(\mathbb{C}_{\ell}) \subseteq \widehat{\Delta}(\mathcal{O})$, one has the usual associated idempotent
$$e_{\psi} = \frac{1}{p-1}\sum_{\sigma \in \Delta}\psi(\sigma) \sigma^{-1} \in \mathcal{O}[\Delta].$$
We consider the $\mathcal{O}[\Delta]$-modules 
$${\rm BF}_{\mathcal{O}}(Y):= \mathcal{O} \otimes_{\mathbb{Z}_{\ell}} {\rm BF}_{\ell}(Y) \text{ and } {\rm Cl}^{-}_{\mathcal{O}}(K):=\mathcal{O} \otimes_{\mathbb{Z}_{\ell}} {\rm Cl}^{-}_{\ell}(K),$$
where ${\rm BF}_{\ell}(Y) = \mathbb{Z}_{\ell} \otimes_{\mathbb{Z}} {\rm BF}(Y)$ and ${\rm Cl}^{-}_{\ell}(K)$ is the Sylow $\ell$-subgroup of ${\rm Cl}^{-}(K)$.  The second main result that we prove in this article is the following.
\begin{lthm}[\cref{equ_card}] \label{main2}
Assume that $\ell$ is a rational prime satisfying (\ref{div_for_intro}).  If $\ell \neq p$ and $\psi \in \widehat{\Delta}(\mathbb{C}_{\ell}) \subseteq \widehat{\Delta}(\mathcal{O})$ is an odd character, then
$$\#e_{\psi}{\rm BF}_{\mathcal{O}}(Y) = \#e_{\psi}{\rm Cl}^{-}_{\mathcal{O}}(K).$$
Otherwise, if $\psi \in \widehat{\Delta}(\mathbb{C}_{p}) \subseteq \widehat{\Delta}(\mathbb{Z}_{p})$ is an odd character different from the Teichm\"{u}ller character $\omega \in \widehat{\Delta}(\mathbb{Z}_{p})$, one has
$$\#e_{\psi} {\rm BF}_{p}(Y) = p \cdot \# e_{\psi} {\rm Cl}^{-}_{p}(K),$$
and we have also $\#e_{\omega} {\rm BF}_{p}(Y) = \#e_{\omega} {\rm Cl}^{-}_{p}(K) = 1$.
\end{lthm}

\subsection*{Outlook} 
In this article, we have chosen to concentrate on establishing a link between the minus class group of a cyclotomic field with prime conductor and the Bowen--Franks group of a digraph.  A motivation for the present work is to establish a foundation for connecting the Iwasawa theory of digraph coverings with the classical Iwasawa theory of number fields. In the latter, one studies the growth of the minus part of the ideal class groups of the cyclotomic fields $\mathbb{Q}(\mu_{p^n})$ as $n \to \infty$. A central tool is the $p$-adic $L$-function and its relationship, via the Iwasawa main conjecture, to the structure of the corresponding Iwasawa modules. This connection allows one to describe the asymptotic behavior of these minus parts in terms of invariants of the $p$-adic $L$-function.
Analogous results have been obtained for undirected graphs in a series of works \cite{Vallieres:2021,McGown/Vallieres:2021,McGown/Vallieres3:2021,Gonet:2021,Gonet:2022,Kleine/Muller:2024}. More recently, some of these results have been extended to the context of digraphs in \cite{Lei/Muller:2026}. We hope that the present work will help establish a more concrete relationship between this emerging theory and the classical one.

\subsection*{Organization}
The article is organized as follows.  In \cref{section:digraphs}, we introduce the relevant notations and concepts pertaining to digraphs such as the Bowen--Franks operator, the Bowen--Franks group, the zeta function, its special value at $u=1$, and we study how these objects behave in covering maps of digraphs.  In \cref{section:equivariant}, we place ourselves in an equivariant setting and study the objects introduced in \cref{section:digraphs}, taking into account the action of a finite abelian group. In \cref{section:voltage}, we review the notion of voltage assignment and derived graph originally due to \cite{Gross:1974}, which gives a convenient way of constructing Galois covers of digraphs.  In \cref{section:special_cover}, using a particular voltage assignment on a bouquet digraph $X$, we introduce the Galois cover $Y$ of $X$ with Galois group $\Delta$, and we obtain our main results connecting the Bowen--Franks group of $Y$ with the minus class groups of $K = \mathbb{Q}(\zeta_{p})$.  The article has two appendices. \cref{gr_acting_on_digraph} presents basic results concerning the Galois theory of coverings of digraphs. These results are likely already well-known to experts. We include full details there because we are not aware of a reference in the literature where they are written down. Finally, in \cref{app:alternative}, we show that the Bowen--Franks operator for the digraph $Y$ can be represented by an explicit circulant matrix. This observation enables us to give a different proof of \cref{thm:m(Y)}, which plays a key role in our proof of \cref{main1}. The proof in \cref{app:alternative} is based on standard properties of circulant matrices and elementary facts about polynomial resultants, which may be of independent interest.

\subsection*{Acknowledgments}
The authors thank Rusiru Gambheera for valuable input and comments during the preparation of this article. DV thanks John Lind, Thomas Mattman, Kevin McGown, and Riccardo Pengo for various stimulating discussions related to the matters of this article. AL's research is supported by the NSERC Discovery Grants Program RGPIN-2026-04351. DV acknowledges support from an AMS-Simons Research Enhancement Grant for PUI Faculty.

%\section{Digraphs}
\section{Digraphs} \label{section:digraphs}
A digraph $X = (V_{X},E_{X})$ consists of a collection of vertices $V_{X}$ and a collection of edges $E_{X}$ with an incidence map
$${\rm inc}:E_{X} \rightarrow V_{X} \times V_{X},$$
which we denote by $e \mapsto {\rm inc}(e) = (o(e),t(e))$.  The vertex $o(e)$ is called the origin vertex of $e$, and $t(e)$ the terminal vertex of $e$.  A digraph $X$ is called finite if both $V_{X}$ and $E_{X}$ are finite sets.  Given $v \in V_{X}$, we let
$$E_{X,v}^{o} = \{e \in E_{X}: o(e) = v \} \text{ and } E_{X,v}^{t} = \{e \in E_{X} : t(e) = v \}. $$
A digraph is called locally finite if $E_{X,v}^{o}$ and $E_{X,v}^{t}$ are finite sets for all $v \in V_{X}$.

A path $c = e_{1} \cdot \ldots \cdot e_{m}$ in a digraph $X$ consists of a finite sequence of edges $e_{i} \in E_{X}$ satisfying $t(e_{i}) = o(e_{i+1})$ for $i=1,\ldots,m-1$.  The length of the path $c$ is $m$, and is denoted by ${\rm len}(c)$.  We let $o(c) = o(e_{1})$ and $t(c) = t(e_{m})$, and a path is called closed if $o(c) = t(c)$. An $n$-multiple of a closed path $c$ is a path obtained by repeating $c$ $n$-times. A path is said to be prime if it is not an $n$-multiple of a closed path with $n
\ge2$. We call two closed paths equivalent if one is obtained from the other by a cyclic permutation of edges. A prime cycle is a class of prime closed paths with respect to this equivalence relation. 

A digraph $X$ is called strongly connected if given any two distinct vertices $v_{1},v_{2} \in V_{X}$, there is a path $c$ in $X$ such that $o(c) = v_{1}$ and $t(c) = v_{2}$.

Let $Y = (V_{Y},E_{Y})$ and $X = (V_{X},E_{X})$ be two digraphs.  A morphism $f:Y \rightarrow X$ of digraphs is a pair of functions
$$f_{V}:V_{Y} \rightarrow V_{X} \text{ and } f_{E}:E_{Y} \rightarrow E_{X} $$
satisfying $f_{V}(o(\varepsilon)) = o(f_{E}(\varepsilon))$ and $f_{V}(t(\varepsilon)) = t(f_{E}(\varepsilon))$ for all $\varepsilon \in E_{Y}$. Quite often, we will also denote both $f_{V}$ and $f_{E}$ simply by $f$.  An isomorphism of digraphs is a morphism $f:Y \rightarrow X$ for which both $f_{V}$ and $f_{E}$ are bijections.  The group of automorphisms of a digraph $Y$ will be denoted by ${\rm Aut}(Y)$.

%\subsection{The Bowen--Franks operator the Bowen--Franks group}
\subsection{The Bowen--Franks operator and the Bowen--Franks group}
Given a digraph $X = (V_{X},E_{X})$, we let $\mathbb{Z}V_{X}$ be the free abelian group on $V_{X}$.  When $X$ is \emph{locally finite}, we consider the adjacency operator $\mathcal{A}_{X}:\mathbb{Z}V_{X} \rightarrow \mathbb{Z}V_{X}$ defined via 
$$v \mapsto \mathcal{A}_{X}(v) = \sum_{\substack{e \in E_{X,v}^{o}}} t(e).$$
The identity operator on $\mathbb{Z}V_{X}$ will be denoted by $\mathcal{I}_{X}$.
\begin{definition}
Let $X$ be a locally finite digraph.  The Bowen--Franks operator $\mathcal{BF}_{X}$ on $\mathbb{Z}V_{X}$ and the Bowen--Franks group ${\rm BF}(X)$ of $X$ are defined to be $\mathcal{BF}_{X} = \mathcal{I}_{X} - \mathcal{A}_{X}$ and ${\rm BF}(X) = {\rm coker}(\mathcal{BF}_{X})$, respectively. 
\end{definition}
Note that if $X$ is a finite digraph, then ${\rm BF}(X)$ is a finitely generated abelian group.  Throughout this article, we will sometimes drop the index $X$ from the notation and write $\mathcal{A}$ and $\mathcal{BF}$, instead of $\mathcal{A}_{X}$ and $\mathcal{BF}_{X}$, if the digraph $X$ is understood from the context.

%\section{The Artin--Mazur zeta function}
\subsection{The Artin--Mazur zeta function}  
Our main reference for this section is \cite{Kotani/Sunada:2000}.  Let $X$ be a strongly connected \emph{finite} digraph.  Its zeta function is defined to be the formal power series
$$Z_{X}(u) = \exp \left(\sum_{m=1}^{\infty} N_{m} \frac{u^{m}}{m} \right) \in 1 + u\mathbb{Q} \llbracket u \rrbracket, $$
where $N_{m}$ is the number of closed paths of length $m$ in $X$. This function is an example of an Artin--Mazur zeta function; it is the Artin--Mazur zeta function associated with the shift of finite type corresponding to the digraph $X$, namely, the edge shift.  By \cite[Theorem 2.3]{Kotani/Sunada:2000}, one has
$$Z_{X}(u) = \prod_{\mathfrak{c}}(1 - u^{\rm len(\mathfrak{c})})^{-1}, $$
where the product is over all prime cycles in $X$. This shows, in particular, that $Z_{X}(u) \in 1 + u \mathbb{Z}\llbracket u \rrbracket$.  Moreover, \cite[Lemma 2.2]{Kotani/Sunada:2000} shows that $Z_{X}(u)$ is a rational function and, more precisely
$$Z_{X}(u)^{-1} = {\rm det}(\mathcal{I}_{X} - \mathcal{A}_{X}u) \in \mathbb{Z}[u].  $$
From now on, we write
\begin{equation} \label{AM_zeta}
g_{X}(u) = {\rm det}(\mathcal{I}_{X} - \mathcal{A}_{X}u) \in \mathbb{Z}[u]. 
\end{equation}
The zeta function $Z_{X}(u)$ is sometimes called the Bowen--Lanford zeta function in the literature, since Bowen and Lanford were the first to notice that it is a rational function in \cite{Bowen/Lanford:1970}.

%\subsection{The special value \texorpdfstring{$u=1$}{} of the zeta function}
\subsection{The special value \texorpdfstring{$u=1$}{} of the zeta function}
Let $F$ be a field of characteristic zero.  Given a polynomial $P(u) \in F[u]$, we can consider its Taylor expansion around $1$, namely
$$P(u) = a_{0} + a_{1}(u-1) + a_{2}(u-1)^{2} + \ldots,$$
for some $a_{i} \in F$.  From now on, we let $P^{*}(1)$ be the first non-vanishing Taylor coefficient at $u=1$ of the polynomial $P(u)$.  \begin{comment}In other words, one has
$$P(u) = P^{*}(1)(u-1)^{r} + \ldots,$$
where $P^{*}(1) \neq 0$, and
$$r = {\rm ord}_{u=1}(P(u))$$
is the order of vanishing of $P(u)$ at $u=1$.  Note that in terms of derivatives, one has
$$P^{*}(1) = \frac{P^{(r)}(1)}{r!}.$$
\end{comment}
In particular, if $X$ is a strongly connected finite digraph, then we let
$$r_{X} = {\rm ord}_{u=1}(g_{X}(u)),$$
and we will be particularly interested in the special value $g_{X}^{*}(1)$.  If $r_{X} = 0$, then $g_{X}^{*}(1) = g_{X}(1)$ is the Parry-Sullivan invariant of the edge shift associated with the finite digraph $X$ introduced in \cite{Parry/Sullivan:1975}.  The Bowen--Franks group and the Parry-Sullivan invariant play an essential role in the classification of irreducible shifts of finite type of positive entropy up to flow equivalence (see \cite{Parry/Sullivan:1975,Bowen/Franks:1977,Franks:1984} and also \cite[\S 13.6]{Marcus/Lind:2021}).

If $V$ is a finite dimensional vector space, $f$ is an $F$-linear operator on $V$, and $\lambda$ is an eigenvalue for $f$, we let $\alpha_{f}(\lambda)$ denote its algebraic multiplicity and $\beta_{f}(\lambda)$ its geometric multiplicity with the understanding that $\alpha_{f}(\lambda) = \beta_{f}(\lambda) = 0$ when $\lambda \in F$ is not an eigenvalue.  Throughout this article, we abbreviate $R \otimes_{\mathbb{Z}} \mathbb{Z}V_{X}$ by $RV_{X}$, whenever $R$ is a unital commutative ring. The following proposition is an elementary fact from linear algebra.
\begin{proposition} \label{ord_of_va}
Let $X$ be a strongly connected finite digraph.  Then,
$$r_{X} =  \alpha_{\mathcal{A}_{F}}(1) \text{ and }  {\rm rank}_{\mathbb{Z}}\,\BF(X) = \beta_{\mathcal{A}_{F}}(1), $$
where $\mathcal{A}_{F}:FV_{X} \rightarrow FV_{X}$ denotes the base change of $\mathcal{A}$ to $F$.
\end{proposition}
As a consequence of \cref{ord_of_va}, we obtain the following corollary.
\begin{corollary} \label{ineq}
If $X$ is a strongly connected finite digraph, then
$$r_{X} \ge {\rm rank}_{\mathbb{Z}}\BF(X). $$
\end{corollary}
The equality holds in some instances, e.g., if $r_X\in \{0,1\}$.  Moreover, we will show later that the equality holds for abelian Cayley--Serre digraphs (see \cref{delta_zero_cs} below).  
On the other hand, there are strongly connected finite digraphs for which the equality does not happen. See for instance \cite[Example 7.2]{Lei/Muller:2026}.  Following \cite[Definition 7.1]{Lei/Muller:2026}, we make the following definition.
\begin{definition}
Let $X$ be a strongly connected finite digraph.  We define
\begin{equation*}
\delta(X) = r_{X} - {\rm rank}_{\mathbb{Z}}{\rm BF}(X) \in \mathbb{Z}_{\ge 0}.
\end{equation*}
\end{definition}

It follows directly from \eqref{AM_zeta} that the zeta function of $X$, as well as its special value $g_{X}^{*}(1)$, can be explicitly given in terms of the eigenvalues of $\mathcal{A}_{F}$:  

\begin{lemma} \label{zeta_explicit}
Let $X$ be a strongly connected finite digraph, and let $F$ be an algebraically closed field of characteristic zero.  Consider the morphism of $F$-vector spaces $\mathcal{A}_{F}:FV_{X} \rightarrow FV_{X}$ obtained from the adjacency operator $\mathcal{A}_{X}$ by extending the scalars to $F$.  Label the eigenvalues  $\lambda_{1},\ldots,\lambda_{n}$ with the understanding that 
$$\lambda_{1}= \ldots = \lambda_{r} = 1,$$ 
and $\lambda_{i} \neq 1$ for $i \notin \{1,\ldots,r \}$.  In other words, $r = r_{X} =  \alpha_{\mathcal{A}_{F}}(1)$.  Then, one has
$$g_{X}(u) = (1-u)^{r} \prod_{i = r+1}^{n}(1 - \lambda_{i}u).$$
Moreover,
$$g_{X}^{*}(1) = (-1)^{r} \prod_{i=r+1}^{n} (1 - \lambda_{i}).$$
\end{lemma}
\begin{comment}
\begin{proof}
It follows from \eqref{AM_zeta} that
\begin{equation*}
\begin{aligned}
g_{X}(u) &= \prod_{i=1}^{n}(1 - \lambda_{i}u) = (1-u)^{r} \cdot \prod_{i=r+1}^{n}(1 - \lambda_{i}u).
\end{aligned}
\end{equation*}
Let 
$$P(u) = (1-u)^{r} \text{ and } Q(u) = \prod_{i=r+1}^{n}(1 - \lambda_{i}u).$$
A repeated application of the Leibniz rule gives 
$$g_{X}^{(r)}(u) = P^{(r)}(u) Q(u) + R(u),$$
for some $R(u) \in F[u]$ satisfying $R(1) = 0$.  It follows that
\begin{equation}
g_{X}^{(r)}(1) = P^{(r)}(1) \cdot Q(1) = P^{(r)}(1) \cdot \prod_{i=r+1}^{n}(1 - \lambda_{i}).
\end{equation}
A simple calculation shows that $P^{(r)}(u) = r! (-1)^{r}$ from which the result follows.
\end{proof}
\end{comment}
The following theorem relates the special value $g_{X}^{*}(1)$ with the cardinality of the torsion part of the Bowen--Franks group assuming $\delta(X) = 0$. 
\begin{theorem} \label{the_special_value}
Let $X$ be a strongly connected finite digraph.  If $\delta(X) = 0$, then there exists a nonzero integer $m(X) \in \mathbb{Z}$ such that
$$g_{X}^{*}(1) =  m(X) \cdot \# {\rm BF}(X)_{\rm tors}.$$
\end{theorem}
\begin{proof}
First, note that the characteristic polynomials of $\mathcal{A}$ and $\mathcal{BF}$ are related via the equality
\begin{equation} \label{char_eq}
\chi_{\raisebox{-0.5ex}{$\scriptstyle \mathcal{A}$}}(t) = (-1)^{n} \cdot \chi_{\raisebox{-0.5ex}{$\scriptstyle \mathcal{BF}$}}(1-t),
\end{equation}
where $n = \# V_{X}$.  Label the vertices of $X$, say $V_{X} = \{v_{1},\ldots,v_{n}\}$, and let $B \in M_{n}(\mathbb{Z})$ be the matrix representing the Bowen--Franks operator $\mathcal{BF}_{X}$ with respect to the ordered basis $(v_{1},\ldots,v_{n})$.  Recall that we have
$$\chi_{\raisebox{-0.5ex}{$\scriptstyle \mathcal{BF}$}}(t) = \sum_{i=0}^{n} (-1)^{i} c_{i}(B) t^{n-i},$$
where $c_{i}(B)$ is the sum of principal minors of size $i$ of $B$, for $i=0,\ldots,n$.  On the one hand, it follows from (\ref{char_eq}) that 
$$g_{X}^{*}(1) = \pm c_{n-r}(B) \in \mathbb{Z},$$
where $r = \alpha_{\mathcal{A}_{F}}(1)$, and on the other hand,
$${\rm Fit}^{s}_{\mathbb{Z}}({\rm BF}(X)) = (\# {\rm BF}(X)_{\rm tors}),$$
where $s = {\rm rank}_{\mathbb{Z}}({\rm BF}(X))$.  By definition,  ${\rm Fit}^{s}_{\mathbb{Z}}({\rm BF}(X))$ is the $(n - s)$-th determinantal ideal of $B$; in other words $\#{\rm BF}(X)_{\rm tors}$ is the greatest common divisor of the $(n-s)\times(n-s)$ minors of $B$.  The assumption $\delta(X)=0$ implies that $r = s$, and thus we have $\#{\rm BF}(X)_{\rm tors} \mid g_{X}^{*}(1)$ as desired.
\end{proof}

\begin{remark}
If $r_{X}=0$, then $\delta(X)=0$ and $m(X) = \pm 1$. Indeed, in this case
$$g_{X}^{*}(1) = g_{X}(1) = \pm {\rm det}(B) = \pm \# {\rm BF}(X)_{\rm tors} = \pm \# {\rm BF}(X),$$
and $g_{X}(1)$ is the Parry-Sullivan invariant of $X$.  Recall that the sign of the Parry-Sullivan invariant plays a role in the classification of irreducible shifts of finite type (see the first paragraph of \cite[page 54]{Franks:1984}).
\end{remark}

The absolute value $|m(X)|$ of the integer $m(X)$ appearing in \cref{the_special_value} can be interpreted as the cardinality of a finite abelian group as we now explain (see \cref{m_as_card} below).  Consider the Bowen--Franks operator $\mathcal{BF}:\mathbb{Z}V_{X} \rightarrow \mathbb{Z}V_{X}$ and its base change to $\mathbb{Q}$, namely 
$$\mathcal{BF}_{\mathbb{Q}}:\mathbb{Q}V_{X} \rightarrow \mathbb{Q}V_{X}.$$
Consider the Fitting decomposition
$$\mathbb{Q}V_{X} = N \oplus W,$$
where $N,W$ are $\mathbb{Q}$-subspaces of $\mathbb{Q}V_{X}$ invariant under $\mathcal{BF}_{\mathbb{Q}}$, and where $\mathcal{BF}_{\mathbb{Q}}$ is nilpotent on $N$ and an automorphism of $\mathbb{Q}$-vector spaces on $W$.  If one assumes that $\delta(X) = 0$, then one has
\begin{equation} \label{simple_fitting_dec}
N = {\rm ker}(\mathcal{BF}_{\mathbb{Q}}) \text{ and } W = {\rm Im}(\mathcal{BF}_{\mathbb{Q}}).
\end{equation}
\begin{definition}\label{def:L(X)}
Let $X$ be a strongly connected finite digraph for which $\delta(X) = 0$, and let $\mathcal{BF}:\mathbb{Z}V_{X} \rightarrow \mathbb{Z}V_{X}$ be the Bowen--Franks operator.  As above, consider the Fitting decomposition
$$\mathbb{Q}V_{X} = {\rm ker}(\mathcal{BF}_{\mathbb{Q}}) \oplus {\rm Im}(\mathcal{BF}_{\mathbb{Q}}) $$
associated with the operator $\mathcal{BF}_{\mathbb{Q}}$.  We define
$$\mathcal{L}(X) = \mathbb{Z}V_{X} \cap W,$$
where we view $\mathbb{Z}V_{X}$ inside of $\mathbb{Q}V_{X}$ via the natural inclusion $\mathbb{Z}V_{X} \hookrightarrow \mathbb{Q}V_{X}$.
\end{definition}
Note that since $\mathcal{L}(X) \leq \mathbb{Z}V_{X}$ and $\mathbb{Z}V_{X}$ is free of finite rank over $\mathbb{Z}$, so is $\mathcal{L}(X)$.
\begin{proposition} \label{inclusions}
Let $X$ be a strongly connected finite digraph for which $\delta(X) = 0$, and let $\mathcal{BF}:\mathbb{Z}V_{X} \rightarrow \mathbb{Z}V_{X}$ be the Bowen--Franks operator.  Then, one has
$$\mathcal{BF}(\mathcal{L}(X)) \subseteq \mathcal{BF}(\mathbb{Z}V_{X}) \subseteq \mathcal{L}(X),$$
and
$${\rm rank}_{\mathbb{Z}}(\mathcal{BF}(\mathcal{L}(X))) = {\rm rank}_{\mathbb{Z}}(\mathcal{BF}(\mathbb{Z}V_{X})) = {\rm rank}_{\mathbb{Z}}(\mathcal{L}(X)).$$
\end{proposition}
\begin{proof}
The inclusions are clear from the definitions, and since $\mathcal{BF}_{\mathbb{Q}}|_{W}$ is injective, we have an injective group morphism
$$\mathcal{L}(X) \hookrightarrow \mathcal{BF}(\mathcal{L}(X))$$
induced by $\mathcal{BF}$.  The equalities of ranks then follow.
\end{proof}

\begin{definition}
Let $X$ be a strongly connected finite digraph for which $\delta(X) = 0$, and let $\mathcal{BF}:\mathbb{Z}V_{X} \rightarrow \mathbb{Z}V_{X}$ be the Bowen--Franks operator.  We define
$$\mathcal{R}(X) = \mathcal{BF}(\mathbb{Z}V_{X})/\mathcal{BF}(\mathcal{L}(X)).$$
\end{definition}
Note that by \cref{inclusions}, the abelian group $\mathcal{R}(X)$ is finite.  We conclude this subsection with the following theorem, which relates the cardinality of $\mathcal{R}(X)$ with the integer $m(X)$ introduced in \cref{the_special_value}.
\begin{theorem} \label{m_as_card}
Let $X$ be a strongly connected finite digraph for which $\delta(X) = 0$.  Then, one has
$$\# \mathcal{R}(X) = |m(X)|.$$
\end{theorem}
\begin{proof}
Since $\mathcal{BF}$ is injective on $\mathcal{L}(X)$, one has that $[\mathcal{L}(X):\mathcal{BF}(\mathcal{L}(X))]$ is the determinant of $\mathcal{BF}_{\mathbb{Q}}|_{W}$, which is precisely the product of the nonzero eigenvalues of $\mathcal{BF}_{\mathbb{Q}}:\mathbb{Q}V_{X} \rightarrow \mathbb{Q}V_{X}$.  Combining with \cref{zeta_explicit}, one gets
\begin{equation} \label{spe_index}
|g_{X}^{*}(1)| = [\mathcal{L}(X):\mathcal{BF}(\mathcal{L}(X))].
\end{equation}
Now, note that by construction, $\mathcal{L}(X)$ is saturated in $\mathbb{Z}V_{X}$ (meaning that $\mathbb{Z}V_{X}/\mathcal{L}(X)$ has no non-trivial $\mathbb{Z}$-torsion).  Moreover, the natural inclusion $\mathcal{L}(X) \hookrightarrow \mathbb{Z}V_{X}$ induces an injective group morphism
$$\mathcal{L}(X)/\mathcal{BF}(\mathbb{Z}V_{X}) \hookrightarrow {\rm BF}(X),$$
and \cref{inclusions} shows that one has in fact an injective group morphism
\begin{equation} \label{iso_nice}
\mathcal{L}(X)/\mathcal{BF}(\mathbb{Z}V_{X}) \hookrightarrow {\rm BF}(X)_{\rm tors}.
\end{equation}
The fact that $\mathcal{L}(X)$ is saturated in $\mathbb{Z}V_{X}$ shows that (\ref{iso_nice}) is an isomorphism of groups, and thus
\begin{equation} \label{sec_eq}
[\mathcal{L}(X):\mathcal{BF}(\mathbb{Z}V_{X})] = \# {\rm BF}(X)_{\rm tors}.
\end{equation}
Putting \cref{inclusions}, \cref{the_special_value}, (\ref{spe_index}), and (\ref{sec_eq}) together gives the desired result.
\end{proof}

%\subsection{Covering of digraphs}
\subsection{Covering of digraphs}
We begin this subsection with the following definition.
\begin{definition} \label{def_cov_map}
A morphism $f:Y \rightarrow X$ of digraphs is called a cover (or a covering map) if the following two conditions are satisfied:
\begin{enumerate}
\item The function $f:V_{Y} \rightarrow V_{X}$ is surjective,
\item The morphism $f:Y \rightarrow X$ is locally bijective, meaning that for every $w \in V_{Y}$, the induced functions
$$f|_{E_{Y,w}^{o}} :E_{Y,w}^{o}\rightarrow E_{X,f(w)}^{o} \text{ and } f|_{E_{Y,w}^{t}} :E_{Y,w}^{t}\rightarrow E_{X,f(w)}^{t}$$
are bijections.
\end{enumerate}
\end{definition}
Often, we refer to a cover $f:Y \rightarrow X$ simply as $Y/X$ without explicitly naming the covering map $f$.  If $f:Y \rightarrow X$ is a cover of strongly connected finite digraphs, the Bowen--Franks groups of $X$ and $Y$ are related to one another via the following lemma.

\begin{lemma} \label{adj_comm}
Let $f:Y \rightarrow X$ be a cover of strongly connected finite digraphs, and consider the natural induced surjective group morphism $\mathbb{Z}V_{Y} \rightarrow \mathbb{Z}V_{X}$, which we denote by the same symbol.  One has a commutative diagram
\begin{equation} \label{com_adj}
\begin{tikzcd}
\mathbb{Z}V_{Y} \arrow[r, "\mathcal{A}_{Y}"] \arrow[d,"f" ']  & \arrow[d,"f"] \mathbb{Z}V_{Y} \\
\mathbb{Z}V_{X} \arrow[r,"\mathcal{A}_{X}"] &  \mathbb{Z}V_{X}
\end{tikzcd}
\end{equation}
from which it follows that $f$ induces a natural surjective group morphism
\begin{equation} \label{tors_gr_morphism}
{\rm BF}(Y) \rightarrow {\rm BF}(X).
\end{equation}
\end{lemma}
\begin{proof}
It suffices to show the commutativity of the diagram.  Given $w \in V_{Y}$, one has
$$f\left(\mathcal{A}_{Y}(w) \right) = f \left(\sum_{\varepsilon \in E_{Y,w}^{o}} t(\varepsilon) \right) = \sum_{\varepsilon \in E_{Y,w}^{o}}t(f(\varepsilon)) = \sum_{e \in E_{X,f(w)}^{o}}t(e) = \mathcal{A}_{X}(f(w)),$$
because of the local bijectivity condition of \cref{def_cov_map}.
\end{proof}

The following theorem describes how certain invariants introduced in the previous section behave with respect to divisibility under coverings.
\begin{theorem} \label{divi_pro}
Let $Y/X$ be a cover of strongly connected finite digraphs, then one has 
$$g^*_X(1)\mid g^*_Y(1).$$ 
If we assume, in addition, that $\delta(X) = \delta(Y) = 0$, then 
$$m(X) \mid m(Y).$$ 
Furthermore, if $r_{X} = r_{Y}$ (still assuming $\delta(X) = \delta(Y) =0$), then the induced map ${\rm BF}(Y)_{\rm tors} \rightarrow {\rm BF}(X)_{\rm tors}$ from (\ref{tors_gr_morphism}) above is surjective, and one has 
$$\#{\rm BF}(X)_{\rm tors} \mid \#{\rm BF}(Y)_{\rm tors}.$$
\end{theorem}
\begin{proof}
Let $f\colon\mathbb{Z} V_Y\to \mathbb{Z} V_X$ be the natural map. As $\Z V_X$ is a free $\Z$-module, $\ker(f)$ is saturated. Fix a decomposition $\mathbb{Z} V_Y=\ker(f)\oplus V'$ of $\mathbb{Z}$-modules. From \cref{adj_comm}, one deduces that the matrix of the adjacency operator $\mathcal{A}_Y$ with respect to this decomposition is of the form
\[\begin{pmatrix}
    A&B\\0&C
\end{pmatrix},\] where $A\in \textup{Mat}_{d\times d}(\mathbb{Z})$ with $d=\mathrm{rank}(\ker(f))$. Note that $\mathbb{Z} V_X\cong \mathbb{Z} V_Y/\ker(f)$. In particular, $\mathcal{A}_X$ acts on $\mathbb{Z} V_X$ via $C$. Let $\mu_1,\dots, \mu_d$ be the eigenvalues of $A$ (over some large enough extension of $\mathbb{Q}$). Note that the complex numbers $\mu_i$ are algebraic integers due to the fact that $A\in \textup{Mat}_{d\times d}(\Z)$. In particular, $\prod_{\mu_i\neq 1}\vert 1-\mu_i\vert \in \Z$.  We denote by $\lambda_i(X)$ and $\lambda_i(Y)$ the eigenvalues of $\mathcal{A}_Y$ and $\mathcal{A}_X$, respectively, with the same convention as in the statement of \cref{zeta_explicit}. The aforementioned lemma then implies
\[\vert g^*_Y(1)\vert =\prod_{i=r_Y+1}^{\# V_Y}\vert 1-\lambda_i(Y)\vert=\prod_{\mu_i\neq 1}\vert 1-\mu_i\vert \prod_{i=r_X+1}^{\# V_X}(1-\lambda_i(X))\vert =\prod_{\mu_i\neq 1}\vert 1-\mu_i\vert \vert g^*_X(1)\vert \]
proving the first assertion.

 Due to \cref{adj_comm}, the surjective group morphism $f:\mathbb{Z}V_{Y} \rightarrow \mathbb{Z}V_{X}$ induces a surjective group morphism $\mathcal{BF}_{Y}(\mathbb{Z}V_{Y}) \rightarrow \mathcal{BF}_{X}(\mathbb{Z}V_{X})$ which, in turn, induces yet another surjective group morphism
$$\mathcal{R}(Y) \twoheadrightarrow \mathcal{R}(X). $$
\cref{m_as_card} implies the divisibility $m(X) \mid m(Y)$, which proves the second assertion.

 Assuming that $r_{X} = r_{Y}$, the surjective group morphism ${\rm BF}(Y) \rightarrow {\rm BF}(X)$ from \cref{adj_comm} induces a surjective group morphism
$${\rm BF}(Y)_{\rm tors} \rightarrow {\rm BF}(X)_{\rm tors}$$
from which the final assertion follows.
\end{proof}

As an example, consider the $2$-to-$1$ cover $f:Y \rightarrow X$ given by

\begin{equation*}
\begin{tikzpicture} [baseline={([yshift=-0.7ex] current bounding box.center)}]
%vertices
\draw[fill=black] (0,0) circle (1pt);
\draw[fill=black] (1,0) circle (1pt);

%edges
\draw (0,0) edge [decoration={markings, mark= at position 0.44 with {\arrow[xscale=-1]{stealth}}},preaction={decorate},bend left=20] (1,0);
\draw (0,0) edge [decoration={markings, mark= at position 0.58 with {\arrow{stealth}}},preaction={decorate}, bend right=20] (1,0);
\draw (0,0) edge [decoration={markings, mark= at position 0.38 with {\arrow{stealth}}},preaction = {decorate}, loop left, in = 155, out = 205,min distance=8mm] (0,0);
\draw (1,0) edge [decoration={markings, mark= at position 0.38 with {\arrow{stealth}}},preaction = {decorate}, loop right, in = 25, out = 335,min distance=8mm] (1,0);
\draw (0,0) edge [decoration={markings, mark= at position 0.60 with {\arrow{stealth}}},preaction = {decorate}, loop left, in = 145, out = 215,min distance=12mm] (0,0);
\draw (1,0) edge [decoration={markings, mark= at position 0.60 with {\arrow{stealth}}},preaction = {decorate}, loop right, in = 35, out = 325,min distance=12mm] (1,0);
\end{tikzpicture} 
\rightarrow
\begin{tikzpicture} [baseline={([yshift=-0.7ex] current bounding box.center)}]
%vertices
\draw[fill=black] (0,0) circle (1pt);

%edges
\draw (0,0) edge [decoration={markings, mark= at position 0.38 with {\arrow{stealth}}},preaction = {decorate}, loop left, in = 155, out = 205,min distance=8mm] (0,0);
\draw (0,0) edge [decoration={markings, mark= at position 0.60 with {\arrow{stealth}}},preaction = {decorate}, loop left, in = 145, out = 215,min distance=12mm] (0,0);
\draw (0,0) edge [decoration={markings, mark= at position 0.38 with {\arrow{stealth}}},preaction = {decorate}, loop right, in = 25, out = 335,min distance=8mm] (0,0);

\end{tikzpicture} 
\end{equation*}
The adjacency matrices are 
$$\begin{pmatrix} 2& 1 \\ 1 & 2 \end{pmatrix} \text{ and } \begin{pmatrix}3 \end{pmatrix},$$
respectively.  Thus, the zeta functions are given by
$$g_{Y}(u) = 1 - 4u + 3u^{2} \text{ and } g_{X}(u) = 1 - 3u, $$
and one has $r_{Y} = 1 \neq 0 = r_{X}$.  On the other hand, a simple calculation shows that $\delta(X) = \delta(Y) = 0$ and $\# {\rm BF}(Y)_{\rm tors} =1$ and $\# {\rm BF}(X)_{\rm tors} = 2$, from which one obtains $m(Y) = 2$ and $m(X) = -1$.  This example shows that if $r_{X} \neq r_{Y}$, then one does not always have the divisibility $\# {\rm BF}(X)_{\rm tors} \mid \# {\rm BF}(Y)_{\rm tors}$.

%\section{Equivariant digraph theory}
\section{Equivariant digraph theory} \label{section:equivariant}

Our approach in this paper relies on Galois theory for digraphs.  If the reader is not familiar with the Galois theory for digraphs, we present a possible approach in \cref{gr_acting_on_digraph}, which is a straightforward adaptation of the usual Galois theory for undirected graphs, as presented in \cite[Sections 5.1 and 5.2]{Sunada:2013} for instance.  

There are two ways of constructing Galois covers of digraphs.  The first way is a top-down approach and is explained in \cref{one_way_corr}.  The second way is a bottom-up approach using the theory of voltage assignments introduced in \cite{Gross:1974}.  This approach is explained in \cref{concrete_construction}, which we will use in \cref{a_special_cover} to construct our digraph of interest.

%\subsection{Groups acting on digraphs}
\subsection{Groups acting on digraphs} \label{groups_acting}
If $f:Y\rightarrow X$ is a covering map of digraphs, we let
\begin{equation} \label{deck_group_def}
{\rm Aut}_{f}(Y/X) = \{\sigma \in {\rm Aut}(Y) : f \circ \sigma = f \}
\end{equation}
denote the group of deck transformations of the cover $f:Y \rightarrow X$; it is a subgroup of ${\rm Aut}(Y)$.  We will often write ${\rm Aut}(Y/X)$ instead of ${\rm Aut}_{f}(Y/X)$.  Recall that a group $G$ acts on a digraph $Y$ if we are given a group morphism
$$G \rightarrow {\rm Aut}(Y). $$
The group of deck transformations ${\rm Aut}(Y/X)$ of a cover of digraphs $Y/X$ acts on $Y$ in a particular way assuming that $Y$ is strongly connected; the key observation is that the induced action on $V_{Y}$ is free.  
\begin{proposition} \label{act_special}
If $Y/X$ is a cover of strongly connected digraphs, then the action of ${\rm Aut}(Y/X)$ on $V_{Y}$ is free.
\end{proposition}
\begin{proof}
This follows directly from \cref{unique_lifting_prop}.  
\end{proof}

If a group $G$ acts on a digraph, then it induces two $G$-sets, for both $V_{Y}$ and $E_{Y}$ are acted upon.  One obtains a quotient digraph $G \backslash Y = (G \backslash V_{Y}, G \backslash E_{Y})$ whose incidence map is induced from the one on $Y$.  Moreover, the natural map $Y \rightarrow G \backslash Y$ is a morphism of digraphs.  From now on, we will write $Y_{G}$ rather than $G \backslash Y$.

\begin{proposition} \label{quotient_const}
If $G$ acts on a digraph $Y$, and the action of $G$ on $V_{Y}$ is free, then the natural morphism $Y \rightarrow Y_{G}$ is a covering map.  
\end{proposition}
\begin{proof}
The proof is routine, and left to the reader. (The fact that the action of $G$ on $V_{Y}$ is assumed to be free is used to show the injectivity part of the local bijectivity condition in \cref{def_cov_map}.)
\end{proof}

In fact, if one assumes moreover that $Y$ in \cref{quotient_const} is strongly connected, then the cover $Y \rightarrow Y_{G}$ is a Galois cover as explained in \cref{one_way_corr}.

%\subsection{The Bowen--Franks group as a Galois module}
\subsection{The Bowen--Franks group as a Galois module}

Let $Y$ be a digraph and $G$ a group acting on $Y$ freely on $V_{Y}$.  Then $\mathbb{Z}V_{Y}$ is a left $\mathbb{Z}[G]$-module, and in fact it is a permutation module.  The freeness of the action implies that the left $\mathbb{Z}[G]$-module $\mathbb{Z}V_{Y}$ is free over $\mathbb{Z}[G]$ of finite rank $\# V_{X}$ provided $X = Y_{G}$ is finite.

Since any $\sigma \in G$ induces a bijection $E_{Y,w}^{o} \stackrel{\approx}{\longrightarrow} E_{Y,\sigma \cdot w}^{o}$ for all $w \in V_{Y}$, it follows that the adjacency operator $\mathcal{A}_{Y}:\mathbb{Z}V_{Y} \rightarrow \mathbb{Z}V_{Y}$ is a morphism of left $\mathbb{Z}[G]$-modules when $Y$ is assumed to be locally finite.  Indeed, one has
\begin{equation*}
\begin{aligned}
\mathcal{A}_{Y}(\sigma \cdot w) &= \sum_{\varepsilon \in E_{Y,\sigma \cdot w}^{o}}t(\varepsilon) \\ 
&= \sum_{\varepsilon \in E_{Y,w}^{o}}t(\sigma \cdot \varepsilon) \\
&= \sigma \cdot \mathcal{A}_{Y}(w),
\end{aligned}
\end{equation*}
for all $\sigma \in G$ and all $w \in V_{Y}$.
\begin{comment}
\begin{equation*}
\begin{aligned}
\mathcal{A}_{Y}(\sigma \cdot w) &= \sum_{\varepsilon \in E_{Y,\sigma \cdot w}^{t}}o(\varepsilon) \\ 
&= \sum_{\varepsilon \in E_{Y,w}^{t}}o(\sigma \cdot \varepsilon) \\
&= \sigma \cdot \mathcal{A}_{Y}(w),
\end{aligned}
\end{equation*}
for all $\sigma \in G$ and all $w \in V_{Y}$.
\end{comment}
Therefore, the Bowen--Franks operator $\mathcal{BF}_{Y}:\mathbb{Z}V_{Y} \rightarrow \mathbb{Z}V_{Y}$ is also a morphism of left $\mathbb{Z}[G]$-modules.  As a consequence, the Bowen--Franks group ${\rm BF}(Y)$ is a left $\mathbb{Z}[G]$-module, and the exact sequence
\begin{equation} \label{finite_free_presentation}
\mathbb{Z}V_{Y} \stackrel{\mathcal{BF}_{Y}}{\longrightarrow} \mathbb{Z}V_{Y} \rightarrow {\rm BF}(Y) \rightarrow 0 
\end{equation}
of left $\mathbb{Z}[G]$-modules gives a finite free presentation of ${\rm BF}(Y)$ over $\mathbb{Z}[G]$ when $Y$ is assumed to be finite.

%\subsection{The equivariant Artin--Mazur zeta function}
\subsection{The equivariant Artin--Mazur zeta function}
Throughout this subsection, we assume that $Y$ is a strongly connected finite digraph and that $G$ is a finite \emph{abelian} group. Our approach is an equivariant one heavily influenced by Bass's approach in \cite{Bass:1992}; the assumption on $G$ being abelian allows us to avoid the use of non-commutative determinants.  

 We assume that $G$ acts on $Y$ such that the action of $G$ on $V_{Y}$ is free. Set $X = Y_{G}$.  We consider the polynomial module $\mathbb{Z}V_{Y}[u]$ which is a $\mathbb{Z}[G][u]$-module, and thus we can view $\mathcal{I}_{Y} - \mathcal{A}_{Y}u$ as an operator on the $\mathbb{Z}[G][u]$-module $\mathbb{Z}V_{Y}[u]$.  Note that $\mathbb{Z}V_{Y}[u]$ is a free $\mathbb{Z}[G][u]$-module of finite rank $\# V_{X}$, where $X = Y_{G}$. Thus, the determinant of $\mathcal{I}_{Y} - \mathcal{A}_{Y}u$ over $\mathbb{Z}[G][u]$ is defined.

\begin{definition} \label{equivariant_def_zeta}
With the notation as above, we define
$$\gamma_{Y/X}(u) = {\rm det}_{\mathbb{Z}[G][u]}(\mathcal{I}_{Y} - \mathcal{A}_{Y}u) \in \mathbb{Z}[G][u]. $$
\end{definition}
Note that 
$$\gamma_{Y/X}(u) \in 1 + u \mathbb{Z}[G][u] \subseteq \mathbb{Z}[G]\llbracket u \rrbracket^{\times}, $$
and therefore $\gamma_{Y/X}(u)$ is invertible as a power series with coefficients in $\mathbb{Z}[G]$.  Its inverse 
$$\gamma_{Y/X}(u)^{-1} \in 1 + u \mathbb{Z}[G]\llbracket u \rrbracket $$
is the equivariant (Artin--Mazur) zeta function associated with the abelian Galois cover $Y/X$ of digraphs, but our emphasis in this article will be on $\gamma_{Y/X}(u)\in 1 + u \mathbb{Z}[G][u]$.  We will be particularly interested in the equivariant special value
\begin{equation} \label{eq_special_value}
\gamma_{Y/X}(1) \in \mathbb{Z}[G]. 
\end{equation}

%\subsection{Artin formalism}
\subsection{Artin formalism}
This section follows \cite{Gambheera/Vallieres:2025}, except that we work over an arbitrary algebraically closed field $F$ of characteristic zero.  As noted by the authors in the paragraph following \cite[(6.7)]{Gambheera/Vallieres:2025}, the arguments of \cite{Gambheera/Vallieres:2025} remain valid in this more general setting.

We let $G$ be a finite abelian group, $H$ a subgroup, and $\Gamma = G/H$.  Let $Y$ be a strongly connected finite digraph on which $G$ acts in such a way that the action is free on the vertices $V_{Y}$. 
\begin{remark}
The hypothesis that the action of $G$ is free on the vertices $V_{Y}$ can be dispensed with at the cost of losing the inflation property (\cref{equiv_inf} below) and the Galois correspondence (\cref{gal_cor}).  This question was investigated in detail for graphs in \cite{Gambheera/Vallieres:2025}.  In the context of the current article, we find it appropriate to maintain that hypothesis, since the inflation property will play an important role, for instance, in \cref{bf_and_cg}.
\end{remark}

Since the action of $G$ is free on $V_{Y}$, so is the action of $H$.  Throughout this subsection, we write $X = Y_{G}$ and $Z = Y_{H}$, with the natural covering maps
$$Y \rightarrow Z \rightarrow X,$$
so that $Z$ with these two covering maps is an intermediate cover of $Y/X$.  \cref{one_way_corr} shows that $G = {\rm Gal}(Y/X)$ and $H = {\rm Gal}(Y/Z)$.  Moreover, note that \cref{galois_top,last_gal_result} implies that both $Y/Z$ and $Z/X$ are Galois.  Furthermore, $\Gamma$ is naturally isomorphic to ${\rm Gal}(Z/X)$ by \cref{last_gal_result}.

We consider the function
\begin{equation} \label{need_for_ind}
N_{\Gamma}:F[G] \rightarrow F[H]
\end{equation}
of \cite[(4.8)]{Gambheera/Vallieres:2025}, which was denoted therein as $N_{G/H}$, and whose definition we now recall.  The module $F[G]$ is a free $F[H]$-module of finite rank, since
$$F[G] \simeq (F[H])^{\# \Gamma}.$$
  Given $\lambda \in F[G]$, let
$$ m_{\lambda}: F[G] \rightarrow F[G]$$
denote the multiplication by $\lambda$ map.  It is a morphism of $F[H]$-modules, and the map $N_{\Gamma}$ from (\ref{need_for_ind}) is defined via
$$N_{\Gamma}(\lambda) = {\rm det}_{F[H]}(m_{\lambda}).$$
Note that $N_{\Gamma}$ is multiplicative, i.e., if $\lambda_{1}, \lambda_{2} \in F[G]$, then $N_{\Gamma}(\lambda_{1} \lambda_{2}) = N_{\Gamma}(\lambda_{1}) N_{\Gamma}(\lambda_{2})$.  Similarly, since $F[G][u]$ is a free $F[H][u]$-module, one also has for any $P \in F[G][u]$ a multiplication by $P$ map
$$m_{P}:F[G][u] \rightarrow F[G][u],$$
inducing a multiplicative function
$$N_{\Gamma}:F[G][u] \rightarrow F[H][u],$$
which we denote by the same symbol.  

\begin{theorem}[Induction property]\label{equiv_ind} 
Let $Y$ be a strongly connected finite digragh on which a finite abelian group $G$ is acting in a way that the induced action of $G$ on $V_{Y}$ is free.  If $H$ is any subgroup of $G$, then one has
$$N_{\Gamma}(\gamma_{Y/X}(u)) = \gamma_{Y/Z}(u), $$
where $Z = Y_{H}$ and $X = Y_{G}$.
\end{theorem}
\begin{proof}
The proof is identical to the proof of \cite[Theorem 4.12]{Gambheera/Vallieres:2025}.
\end{proof}

Consider now the natural projection map $\pi:F[G] \rightarrow F[\Gamma]$, which is an $F$-algebra morphism.  It induces yet another morphism
\begin{equation} \label{pi_map}
\pi:F[G][u] \rightarrow F[\Gamma][u] 
\end{equation}
of $F[u]$-algebra, which we denote by the same symbol.  
\begin{theorem}[Inflation property] \label{equiv_inf}
Let $Y$ be a strongly connected finite digraph on which a finite abelian group $G$ is acting in a way that the induced action of $G$ on the vertices $V_{Y}$ is free.  If $H$ is any subgroup of $G$, then one has
$$\pi(\gamma_{Y/X}(u)) = \gamma_{Z/X}(u),$$
where $\pi:F[G][u] \rightarrow F[\Gamma][u]$ is the map (\ref{pi_map}) above, $Z = Y_{H}$, and $X = Y_{G}$.
\end{theorem}
\begin{proof}
Consider the natural covering map $q: Y \rightarrow Z$; it induces a natural $F[G]$-module morphism $q:F V_Y \rightarrow F V_Z$ which we also denote by the same symbol.  Now, notice that for each $w \in V_{Y}$, we have
$$q(\mathcal{A}_Y(w))=\sum_{\substack{\varepsilon \in E_{Y,w}^{o}}} t(q(\varepsilon))=\sum_{\substack{e \in E_{Z,q(w)}^{o}}} t(e)=\mathcal{A}_Z(q(w)),$$
so that the following diagram
\begin{equation} \label{com1}
\begin{tikzcd}
F V_Y \arrow[r,"\mathcal{A}_Y"] \arrow[d,"q"]  & \arrow[d,"q"] F V_Y \\
F V_Z \arrow[r,"\mathcal{A}_Z"] & F V_Z
\end{tikzcd}
\end{equation}
of $F[G]$-modules commutes.  We let 
$$\Delta_{Y}(u) = \mathcal{I}_{Y} - \mathcal{A}_{Y}u \in {\rm End}_{F[G][u]}(FV_{Y}[u]) \text{ and } \Delta_{Z}(u) = \mathcal{I}_{Z} - \mathcal{A}_{Z}u \in {\rm End}_{F[\Gamma][u]}(FV_{Z}[u]).$$
The commutative diagram (\ref{com1}) induces yet another commutative diagram
\begin{equation} \label{com2}
\begin{tikzcd}[column sep=huge]
F[\Gamma][u]\otimes_{F[G][u]}FV_{Y}[u] \arrow[r,"{\rm id}\otimes \Delta_{Y}(u)"] \arrow[d]  & \arrow[d] F[\Gamma][u]\otimes_{F[G][u]}FV_{Y}[u] \\
F V_Z[u] \arrow[r,"\Delta_{Z}(u)"] & F V_Z[u]
\end{tikzcd}
\end{equation}
of $F[\Gamma][u]$-modules, where the two vertical arrows are isomorphisms.  The usual base change property for determinants combined with the diagram (\ref{com2}) gives
\begin{equation*}
\begin{aligned}
\pi(\gamma_{Y/X}(u)) &= \pi \left({\rm det}_{F[G][u]}(\Delta_{Y}(u)) \right)\\
&= {\rm det}_{F[\Gamma][u]}({\rm id} \otimes \Delta_{Y}(u)) \\
&= {\rm det}_{F[\Gamma][u]}(\Delta_{Z}(u)) \\
&= \gamma_{Z/X}(u),
\end{aligned}
\end{equation*}
as desired.
\begin{comment}\DV{Here.  Ok, we have to decide if we work over $\mathbb{Z}$ or $F$.  When we take character below, we have no choice but working with $F$, but in this subsection, I guess we have a choice.}
This, in turn, induces the following commutative diagram of $\mathbb{Z}[G][u]$-modules.
\begin{equation*}
\begin{tikzcd}[column sep=huge]
\mathbb{Z} V_Y [u]\arrow[r,"\mathcal{I}_{Y} - \mathcal{A}_{Y}u"] \arrow[d,"q"]  & \arrow[d,"q"] \mathbb{Z} V_Y [u]\\
\mathbb{Z} V_Z [u]\arrow[r,"\mathcal{I}_{Z} - \mathcal{A}_{Z}u"] & \mathbb{Z} V_Z [u]
\end{tikzcd}
\end{equation*}
Here, the vertical maps, which we also denote by $q$, are the natural maps induced by those in the previous diagram. Moreover, notice that $\mathbb{Z} V_Y[u]\cong \bigoplus_{i=1}^n \mathbb{Z}[G][u]\cdot w_i\cong \mathbb{Z}[G][u]^n$ and $\mathbb{Z} V_Z[u]\cong \bigoplus_{i=1}^n \mathbb{Z}[\Gamma][u]\cdot v_i\cong \mathbb{Z}[\Gamma][u]^n$. Under this identification $q$ is the same as the natural $\mathbb{Z}[G][u]$-module morphism induced by $\pi$. Therefore, we have
$$\pi({\rm det}_{\mathbb{Z}[G][u]}(\mathcal{I}_{Y} - \mathcal{A}_{Y}u))={\rm det}_{\mathbb{Z}[H][u]}(\mathcal{I}_{Z} - \mathcal{A}_{Z}u).$$
\end{comment}
\end{proof}

\begin{remark}
\cref{equiv_ind,equiv_inf} give equivariant ways of thinking about the induction and inflation properties of the Artin formalism. Compare, for example, with \cite[Proposition 1.8, page 87]{Tate:1984}.
\end{remark}

%\subsection{\texorpdfstring{$L$}{}-functions}
\subsection{\texorpdfstring{$L$}{}-functions}
Once more, we follow \cite{Gambheera/Vallieres:2025} closely, but we work with an algebraically closed field $F$ of characteristic zero.  The finite group $G$ is still assumed to be abelian throughout.  If $V$ is a finite dimensional $F$-linear representation of $G$, we let 
$$\sigma_{V}:F[G][u] \rightarrow F[u]$$
be the function defined as follows  (see \cite[Section 5.2 and Remark 5.3]{Gambheera/Vallieres:2025}):
\[
\begin{tikzcd}
    F[G][u]\arrow[r,"\rho_V"]&\textup{End}_{F}(V)[u]\arrow[r,"\cong"]&\textup{End}_{F[u]}(V[u])\arrow[r,"\det"] &F[u],
\end{tikzcd}
\]
where $\rho_V$ is the natural map induced by the representation $\rho\colon G\to \textup{End}_F(V)$ and the middle isomorphism is given by \cite[Theorem 2.1]{Gambheera/Vallieres:2025}.  The function $\sigma_{V}$ is multiplicative, meaning that
$$\sigma_{V}(P(u) \cdot Q(u)) = \sigma_{V}(P(u)) \cdot \sigma_{V}(Q(u))$$
for all $P(u),Q(u) \in F[G][u]$.  Moreover, \cite[Proposition 5.4]{Gambheera/Vallieres:2025} shows that if $V_{1}$ and $V_{2}$ are two isomorphic finite dimensional $F$-linear representations of $G$, then
\begin{equation}\label{well_defined}
\sigma_{V_{1}} = \sigma_{V_{2}}.
\end{equation}

\begin{definition} \label{def_L_function}
Let $Y/X$ be a Galois cover of finite digraphs with a (finite) abelian group of deck transformations $G = {\rm Gal}(Y/X)$.  Let $V$ be a finite dimensional $F$-linear representation of the finite abelian group $G$.  Then, we define 
$$g_{V}(u) = \sigma_{V}(\gamma_{Y/X}(u)) \in 1 + u F[u].$$
\end{definition}
Note that $g_{V}(u)\in 1 + u F[u] \subseteq F\llbracket u \rrbracket^{\times}$ so that $g_{V}(u)$ is invertible in $F\llbracket u \rrbracket$.  Its inverse $g_{V}(u)^{-1} \in F\llbracket u \rrbracket$ is the (Artin--Mazur) $L$-function associated with $V$.  In this article, we will be mainly concerned with the polynomial $g_{V}(u)$.  We first show that this polynomial depends only on the isomorphism class of $V$.

\begin{theorem}
With the notation as above, let $V_{1}$ and $V_{2}$ be two isomorphic finite dimensional $F$-linear representations of $G$.  Then
$$g_{V_{1}}(u) = g_{V_{2}}(u).$$
\end{theorem}
\begin{proof}
This follows directly from (\ref{well_defined}).
\end{proof}

\cref{additivity,induction,inflation} below show that the polynomials $g_{V}(u)$ satisfy the usual Artin formalism.  (See for instance, \cite[page 15]{Tate:1984}.)

\begin{theorem}[Additivity] \label{additivity}
Let $V_{1}$ and $V_{2}$ be two finite dimensional $F$-linear representations of $G$, then
$$g_{V_{1} \oplus V_{2}}(u) = g_{V_{1}}(u) \cdot g_{V_{2}}(u).$$
\end{theorem}
\begin{proof}
See \cite[Theorem 5.14]{Gambheera/Vallieres:2025}.
\end{proof}

\begin{theorem}[Induction] \label{induction}
With the notation as above, let $V$ be a finite dimensional $F$-linear representation of $H = {\rm Gal}(Y/Z)$, and let ${\rm Ind}(V)$ denote the induced representation from $H$ to $G = {\rm Gal}(Y/X)$.  Then,
$$g_{{\rm Ind}(V)}(u) = g_{V}(u).$$
\end{theorem}
\begin{proof}
This follows from \cref{equiv_ind} just as in the proof of \cite[Theorem 5.15]{Gambheera/Vallieres:2025}.
\end{proof}

\begin{theorem}[Inflation] \label{inflation}
With the notation as above, let $V$ be a finite dimensional representation of $\Gamma = {\rm Gal}(Z/X)$, and let ${\rm Inf}(V)$ be its inflation to $G$, namely the representation of $G$ obtained from the natural projection map $\pi:F[G] \rightarrow F[\Gamma]$.  Then,
$$g_{{\rm Inf}(V)}(u) = g_{V}(u).$$ 
\end{theorem}
\begin{proof}
For simplicity, let $W = {\rm Inf}(V)$.  Consider the $F[u]$-algebra morphism 
$$\rho_{V}:F[\Gamma][u] \rightarrow {\rm End}_{F}(V)[u]$$ 
induced by the $F$-algebra morphism $\rho_{V}:F[\Gamma] \rightarrow {\rm End}_{F}(V)$ associated with the representation $V$.  Then, note that $\rho_{W} = \rho_{V} \circ \pi$, where $\pi$ is the map defined in (\ref{pi_map}) above.  It follows from \cite[(5.1)]{Gambheera/Vallieres:2025} that $\sigma_{V} \circ \pi = \sigma_{W}$, and the result follows from \cref{def_L_function} and \cref{equiv_inf}.
\end{proof}

As before, let $Y/X$ be a Galois cover of finite digraphs with abelian Galois group $G = {\rm Gal}(Y/X)$.  If $\psi \in \widehat{G}(F)$ corresponds to a one dimensional irreducible $F$-linear representation $V$ of $G$, then it induces an $F$-algebra morphism $\psi:F[G] \rightarrow F$ which in turn induces an $F[u]$-algebra morphism
\begin{equation} \label{apply_to_coef}
F[G][u]\rightarrow F[u],
\end{equation}
obtained by applying $\psi$ to the coefficients of a polynomial in $F[G][u]$.  We denote the map (\ref{apply_to_coef}) by the same symbol $\psi$.  \cite[Proposition 5.5]{Gambheera/Vallieres:2025} shows that one has
\begin{equation*} 
\psi(\gamma_{Y/X}(u)) = g_{V}(u),
\end{equation*}
where $\psi$ denotes the map (\ref{apply_to_coef}) above.  For that reason, if $\psi \in \widehat{G}(F)$ corresponds to a one-dimensional irreducible $F$-linear representation $V$ of $G$, then we let
\begin{equation} \label{L_characters}
g_{Y/X}(u,\psi) = \psi(\gamma_{Y/X}(u)) = g_{V}(u).
\end{equation}

\begin{theorem} \label{product_dec}
Let $Y/X$ be a Galois cover of finite digraphs with abelian Galois group $G = {\rm Gal}(Y/X)$, and let $F$ be an algebraically closed field of characteristic zero.  Then, one has
$$g_{Y}(u) = g_{X}(u)  \prod_{\substack{\psi \in \widehat{G}(F) \\ \psi \neq \psi_{0}}} g_{Y/X}(u,\psi).$$
\end{theorem}
\begin{proof}
This is a consequence of the Artin formalism.  Consider the trivial representation $V_{0}$ of the trivial subgroup $\{1\}$ of $G$, and set $V = {\rm Ind}(V_{0})$ to be the induced representation from $H$ to $G$.  Then \cref{induction} implies that
$$g_{{\rm Ind}(V_{0})}(u) = g_{V_{0}}(u) = g_{Y}(u).$$
Since ${\rm Ind}(V_{0}) \simeq F[G]$, \cref{additivity} combined with (\ref{L_characters}) give
\begin{equation*}
\begin{aligned}
g_{Y}(u) &= \prod_{\psi \in \widehat{G}(F)} g_{F[G]e_{\psi}}(u) \\
&= g_{Y/X}(u,\psi_{0})  \prod_{\substack{\psi \in \widehat{G}(F) \\ \psi \neq \psi_{0}}} g_{Y/X}(u,\psi).
\end{aligned}
\end{equation*}
 Finally, \cref{inflation} shows that $g_{Y/X}(u,\psi_{0}) = g_{X}(u)$, which completes the proof.
\end{proof}

Given an abelian Galois cover $Y/X$ of finite digraphs with Galois group $G = {\rm Gal}(Y/X)$, for each $\psi \in \widehat{G}(F)$, we define
$$ r(\psi) = {\rm ord}_{u=1}(g_{Y/X}(u,\psi)).$$
Note that \cref{inflation} implies the equality
$$r(\psi_{0}) = r_{X}.$$

\begin{corollary} \label{decomp}
Let $Y/X$ be a Galois cover of finite digraphs with abelian Galois group $G$.  Then, one has
$$g_{Y}^{*}(1) = g_{X}^{*}(1)  \prod_{\substack{\psi \in \widehat{G}(F) \\ \psi \neq \psi_{0}}} g_{Y/X}^{*}(1,\psi),$$
and
$$r_{Y} = r_{X} + \sum_{\substack{\psi \in \widehat{G}(F) \\ \psi \neq \psi_{0}}} r(\psi).$$
\end{corollary}
\begin{proof}
This is a direct consequence of \cref{product_dec}.
\end{proof}
The following corollary can be viewed as an analogue of the analytic class number formula in algebraic number theory.
\begin{corollary} \label{spe_decomp}
Let $Y/X$ be a Galois cover of finite digraphs with abelian Galois group $G$ for which $\delta(X) = \delta(Y)  = 0$.  Then, one has
$$\frac{m(Y)\cdot  \#{\rm BF}(Y)_{\rm tors}}{m(X)  \cdot \#{\rm BF}(X)_{\rm tors}} = \prod_{\substack{\psi \in \widehat{G}(F) \\ \psi \neq \psi_{0}}} g_{Y/X}^{*}(1,\psi). $$
\end{corollary}
\begin{proof}
This follows directly from \cref{the_special_value,decomp}.
\end{proof}

\section{Voltage assignments} \label{section:voltage}
In this section, we recall how to construct Galois covers of digraphs using voltage assignments in the category of digraphs.  Our main references for this section are \cite{Gross:1974,Gross/Tucker:1977,Gross/Tucker:2001}.  This is the bottom-up approach alluded to above at the beginning of \cref{section:equivariant}. 

%\subsection{Galois covers obtained from voltage assignments}
\subsection{Galois covers obtained from voltage assignments} \label{volt_assignment}
Starting with a strongly connected digraph $X$ and a group $G$, recall that a voltage assignment is simply a function $\alpha:E_{X} \rightarrow G$.  To such a function, one can associate another digraph $Y = X(G,\alpha)$, called the derived digraph, whose construction is as follows.  Its collection of vertices and edges are
$$V_{Y} = V_{X} \times G \text{ and } E_{Y} = E_{X} \times G, $$
respectively.  The incidence map is given by 
$$o(e,\sigma) = (o(e),\sigma) \text{ and } t(e,\sigma) = (t(e),\sigma \alpha(e)). $$
One checks that the natural functions
$$(v,\sigma) \mapsto v \text{ and } (e,\sigma) \mapsto e $$
induce a morphism of digraphs $f:X(G,\alpha) \rightarrow X$, which is a cover.

For $\tau \in G$, let $\phi(\tau):X(G,\alpha) \rightarrow X(G,\alpha)$ be given by
$$(v,\sigma)\mapsto \phi(\tau)(v,\sigma) = (v,\tau \sigma) \text{ and } (e,\sigma) \mapsto \phi(\tau)(e,\sigma) = (e,\tau \sigma).$$
We leave it to the reader to check that 
\begin{equation} \label{iso}
\phi:G \rightarrow {\rm Aut}(X(G,\alpha)/X)
\end{equation}
given by $\tau \mapsto \phi(\tau)$ is a well-defined injective group homomorphism.  It follows that $G$ acts on $X(G,\alpha)$, and moreover the action of $G$ on each of the fibers $f^{-1}(v)$ for all $v \in V_{X}$ is transitive.  The action of $G$ on the vertices of $X(G,\alpha)$ is also free. 

If $c = e_{1} \cdot \ldots \cdot e_{m}$ is any path in $X$, then we set
$$\alpha(c) = \alpha(e_{1}) \ldots \alpha(e_{m}) \in G.$$
Given $v_{0} \in X$, we let $\pi_{1}(X,v_{0})$ be the monoid consisting of closed paths $\gamma$ in $X$ based at $v_{0}$.  The monoid operation is simply concatenation, and the identity is the empty path based at $v_{0}$.  A function $\alpha:E_{X} \rightarrow G$ induces a monoid morphism
\begin{equation} \label{induce_pi_1_map}
\rho_{\alpha}:\pi_{1}(X,v_{0}) \rightarrow G,
\end{equation}
defined via $\rho_{\alpha}(\gamma) = \alpha(\gamma)$ whenever $\gamma \in \pi_{1}(X,v_{0})$.  The following proposition is adapted from the corresponding statement for graphs (see \cite[Theorem 2.11]{Ray/Vallieres:2025} and also \cite[Lemma 2.5]{Lei/Muller:2026}).

\begin{proposition} \label{connectedness}
Let $X$ be a strongly connected digraph and $G$ be a group.  Let also $\alpha:E_{X} \rightarrow G$ be a voltage assignment and let $v_{0} \in V_{X}$.  The digraph $X(G,\alpha)$ is strongly connected if and only if the monoid morphism $\rho_{\alpha}:\pi_{1}(X,v_{0}) \rightarrow G$ defined above in (\ref{induce_pi_1_map}) is surjective. 
\end{proposition}
\begin{proof}
Assume first that $X(G,\alpha)$ is strongly connected.  Then, for any $\sigma \in G$, there is a closed path $c$ in $X(G,\alpha)$ going from $(v_{0},1_{G})$ to $(v_{0},\sigma)$. The path $c$ is of the form
$$(e_{1},1_{G})\cdot(e_{2},\alpha(e_{1}))\cdot \ldots \cdot(e_{m},\alpha(e_{1})\ldots \alpha(e_{m-1})),$$
where $o(e_{1}) = v_{0} = t(e_{m})$, $t(e_{i}) = o(e_{i+1})$ for $i=1,\ldots,m-1$, and $\sigma = \alpha(e_{1})\ldots\alpha(e_{m})$.  Thus, the closed path $\gamma = f(c) = e_{1}\cdot \ldots \cdot e_{m}$ satisfies $\rho_{\alpha}(\gamma) = \sigma$ which shows that $\rho_{\alpha}$ is surjective.

For the converse, we show first that given $(v_{0},\sigma), (v_{0},\tau) \in f^{-1}(v_{0})$, there exists a path in $X(G,\alpha)$ going from $(v_{0},\sigma)$ to $(v_{0},\tau)$.  Indeed, the hypothesis shows that there exists $\gamma \in \pi_{1}(X,v_{0})$ such that $\rho_{\alpha}(\gamma) = \sigma^{-1} \tau$.  One has $\gamma = e_{1} \cdot \ldots \cdot e_{m}$ for some $e_{i} \in E_{X}$.  Then, the path
$$(e_{1},\sigma)\cdot(e_{2},\sigma \alpha(e_{1}))\cdot \ldots \cdot (e_{m},\sigma \alpha(e_{1}) \ldots \alpha(e_{m-1}))$$
connects $(v_{0},\sigma)$ to $(v_{0},\tau)$.  Next, say $w_{1},w_{2}$ are two arbitrary vertices of $X(G,\alpha)$, and assume that $v_{i} = f(w_{i})$ for $i=1,2$.  Since $X$ is assumed to be strongly connected, there exists a path in $X$ going from $v_{1}$ to $v_{0}$.  Since $f:X(G,\alpha) \rightarrow X$ is a cover, this path can be lifted to a path $c_{1}$ in $X(G,\alpha)$ starting at $w_{1}$ and ending at a vertex $w_{0}'$ of $X(G,\alpha)$ lying above $v_{0}$.  Similarly, one can lift a path in $X$ starting at $v_{0}$ and ending at $v_{2}$ to a path $c_{2}$ in $X(G,\alpha)$ starting at $w_{0}'$ and ending at a vertex $w_{2}'$ of $X(G,\alpha)$ lying above $v_{2}$.  Using (\ref{iso}), there exists $\sigma \in G$ such that $\sigma \cdot w_{2}' = w_{2}$.  The path $\sigma \cdot c_{2}$ in $X(G,\alpha)$ starts now at $\sigma \cdot w_{0}'$ and ends at $w_{2}$.  We previously showed that there exists a path $c$ in $X(G,\alpha)$ that connects $w_{0}'$ to $\sigma \cdot w_{0}'$, since both of these vertices lie above $v_{0}$.  The path $c_{1} \cdot c \cdot (\sigma c_{2})$ is now a path in $X(G,\alpha)$ connecting $w_{1}$ to $w_{2}$, which concludes the proof.
\end{proof}

Under a strongly connectedness condition, the digraph $X(G,\alpha)$ gives a convenient way to construct Galois covers, as the following proposition shows.
\begin{proposition} \label{concrete_construction}
Let $X$ be a strongly connected digraph, $G$ a group, $\alpha:E_{X} \rightarrow G$ a function, and assume that $X(G,\alpha)$ is strongly connected.  Then, the cover $f:X(G,\alpha) \rightarrow X$ is a Galois cover, and one has a natural isomorphism
$$G \stackrel{\simeq}{\longrightarrow} {\rm Gal}(X(G,\alpha)/X).$$
\end{proposition}

\begin{proof}
We claim that the injective group morphism $\phi$ from (\ref{iso}) above is surjective.  Indeed, if $\phi \in {\rm Aut}(X(G,\alpha)/X)$, then since $\phi((v,1_{G})) \in f^{-1}(v) = \{(v,\sigma): \sigma \in G \}$, we have that there exists $\sigma_{0} \in G$ such that $\phi((v,1_{G})) = (v,\sigma_{0})$.  On the other hand, $\phi(\sigma_{0})((v,1_{G})) = (v,\sigma_{0})$ as well.  Since $X(G,\alpha)$ is assumed to be strongly connected, \cref{unique_lifting_prop} shows the equality $\phi = \phi(\sigma_{0})$, from which it follows that (\ref{iso}) is a group isomorphism.  The isomorphism (\ref{iso}) can also be used to show that ${\rm Aut}(X(G,\alpha)/X)$ acts transitively on each of the fibers $f^{-1}(v)$; therefore, $X(G,\alpha)/X$ is a Galois cover and this concludes the proof.
\end{proof}

The construction above is functorial meaning that if $\pi:G \rightarrow G'$ is a group morphism, then one obtains a natural morphism of graphs
$$X(G,\alpha) \rightarrow X(G',\pi \circ \alpha),$$
given by $(v,\sigma) \mapsto (v,\pi(\sigma))$ and $(e,\sigma) \mapsto (e,\pi(\sigma))$, and we leave it to the reader to check that if $\pi$ is assumed, moreover, to be surjective, then one obtains a covering map.   

Assume now that $X$ is a strongly connected digraph, $G$ is a group, and $H$ is a subgroup of $G$.  Assume also that $\alpha: E_{X}\rightarrow G$ is any function.  The construction at the beginning of this subsection can be modified accordingly to obtain a derived digraph which we denote by $Y = X(G,H,\alpha)$.  Its sets of vertices and edges are
$$V_{Y} = V_{X} \times (H \backslash G) \text{ and } E_{Y} = E_{X} \times (H \backslash G),$$
respectively.  The incidence map is given by
$$o(e,H \sigma) = (o(e),H \sigma) \text{ and } t(e,H \sigma) = (t(e),H \sigma \alpha(e)).$$
We leave it to the reader to check that the morphism $X(G,\alpha) \rightarrow X(G,H,\alpha)$ given by
$$(v,\sigma) \mapsto (v,H \sigma) \text{ and } (e,\sigma) \mapsto (e,H \sigma) $$
is a covering map, as well as the morphism $X(G,H,\alpha) \rightarrow X$ given by
$$(v,H \sigma) \mapsto v \text{ and } (e,H \sigma) \mapsto e. $$
In other words, the digraph $X(G,H,\alpha)$, with the covering maps $X(G,\alpha) \rightarrow X(G,H,\alpha)$ and $X(G,H,\alpha) \rightarrow X$, is an intermediate cover of $X(G,\alpha) \rightarrow X$.

\begin{remark}
Just as the digraph $X(G,\alpha)$ can be viewed as a slight variant and generalization of the usual notion of Cayley graph, the digraph $X(G,H,\alpha)$ can be viewed as a slight variant and generalization of the notion of Schreier coset graph in combinatorial group theory, but adapted to the category of digraphs.
\end{remark}

The proof of the following proposition is routine and left to the reader.  
\begin{proposition} \label{identification}
Let $X$ be a strongly connected digraph.  Let also $G$ be a group, $H$ a subgroup of $G$ and as above let $\alpha:E_{X}\rightarrow G$ be a voltage assignment.  There is a natural isomorphism
\begin{equation} \label{compatible_two_versions}
X(G,\alpha)_{H} \stackrel{\simeq}{\longrightarrow} X(G,H,\alpha)
\end{equation}
of digraphs that makes both intermediate covers $X(G,\alpha)_{H}$ and $X(G,H,\alpha)$ of $X(G,\alpha)/X$ equivalent ones.  If moreover $H \trianglelefteq G$, then in fact 
$$X(G,H,\alpha) = X(\Gamma,\pi \circ \alpha),$$
where here $\pi:G \rightarrow \Gamma = G/H$ is the natural projection map.
\end{proposition}

Note that when one applies (\ref{compatible_two_versions}) to the case where $H = G$, one gets a digraph isomorphism
$$X(G,\alpha)_{G} \stackrel{\simeq}{\longrightarrow} X.$$

%\subsection{Equivariant Artin--Mazur zeta function and voltage assignments}
\subsection{The equivariant Artin--Mazur zeta function and voltage assignments}
In this subsection, we let $X$ be a strongly connected finite digraph, $G$ a finite abelian group, and we let $\alpha:E_{X} \rightarrow G$ be a voltage assignment.  We fix a labeling
$$V_{X} = \{v_{1},\ldots,v_{n} \}$$
of the vertices of $X$.  Moreover, for simplicity we let $Y = X(G,\alpha)$, and we assume that $Y$ is strongly connected so that we have a Galois cover $Y/X$ by \cref{concrete_construction}.

\begin{proposition}\label{explicit_adjacency_op}
With the notation as above, for each $i=1,\ldots,n$, let $w_{i} = (v_{i},1_{G}) \in V_{Y}$.  Then, for each $j=1,\ldots,n$, one has
$$\mathcal{A}_{Y}(w_{j}) = \sum_{i=1}^{n} \lambda_{ij}w_{i},$$
where
$$\lambda_{ij} = \sum_{\substack{e \in E_{X} \\ {\rm inc}(e) = (v_{j},v_{i})}} \alpha(e) \in \mathbb{Z}[G].$$
\end{proposition}
\begin{proof}
We calculate
\begin{equation*}
\mathcal{A}_{Y}(w_{j}) = \sum_{e \in E_{X,v_{j}}^{o}} t((e,1_{G})) 
= \sum_{e \in E_{X,v_{j}}^{o}} (t(e),\alpha(e)) 
= \sum_{e \in E_{X,v_{j}}^{o}} \alpha(e) (t(e),1_{G}) 
= \sum_{i=1}^{n} \lambda_{ij} w_{i},
\end{equation*}
as we wanted to show.
\end{proof}

As a consequence, we obtain an explicit description of the equivariant zeta function of an abelian Galois cover $X(G,\alpha)/X$ arising from a voltage assignment $\alpha:E_{X} \rightarrow G$.

\begin{corollary}
Let $X$ be a strongly connected finite digraph and let $\alpha:E_{X} \rightarrow G$ be a voltage assignment with values in a finite abelian group $G$.  Assume that $X(G,\alpha)$ is strongly connected, so that we have an abelian Galois cover $X(G,\alpha)/X$, and for simplicity let $Y = X(G,\alpha)$.  Consider the matrix
$$A_{\alpha} = (\lambda_{ij}) \in M_{n}(\mathbb{Z}[G]),$$
where
$$\lambda_{ij} = \sum_{\substack{e \in E_{X} \\ {\rm inc}(e) = (v_{j},v_{i})}} \alpha(e) \in \mathbb{Z}[G].$$
Then, one has
$$\gamma_{Y/X}(u) = {\rm det}(I - A_{\alpha} u) \in 1 + u \mathbb{Z}[G][u].$$
\end{corollary}
\begin{proof}
This follows directly from \cref{equivariant_def_zeta} and \cref{explicit_adjacency_op}.
\end{proof}

In the special case where $X$ is a bouquet digraph (meaning a digraph with a single vertex), one obtains
$$\gamma_{Y/X}(u) = 1 - \sum_{e \in E_{X}} \alpha(e)u \in 1+ u\mathbb{Z}[G][u],$$
and thus
\begin{equation} \label{what_we_need}
\gamma_{Y/X}(1) = 1 - \sum_{e \in E_{X}} \alpha(e) \in \mathbb{Z}[G].
\end{equation}

%\subsection{Cayley--Serre digraphs}
\subsection{Cayley--Serre digraphs}
For (undirected) graphs, the Galois covers of bouquet graphs are called Cayley--Serre graphs (see \cite[page 69]{Sunada:2013}).  Similarly, we shall call the Galois covers of bouquet digraphs (i.e. a digraph with a single vertex) Cayley--Serre digraphs.  Examples of such digraphs can be constructed using voltage assignments on bouquet digraphs as explained in \cref{volt_assignment}.

\begin{proposition} \label{delta_zero_cs}
Let $X = B_{\delta}$ be a bouquet digraph with $\delta$ edges and let $\alpha:E_{X} \rightarrow G$ be a voltage assignment for some finite abelian group $G$.  Let $Y=X(G,\alpha)$ and assume that $Y$ is strongly connected.  For simplicity, write $\mathcal{A}$ instead of $\mathcal{A}_{Y}$.  Let $F$ be an algebraically closed field of characteristic zero.  Then, the base change $\mathcal{A}_{F}:FV_{Y} \rightarrow FV_{Y}$ is diagonalizable and in particular, one has $\delta(Y) = 0$. 
\end{proposition}
\begin{proof}
Let us denote the unique vertex of $X$ by $v_{1}$.  \cref{explicit_adjacency_op} shows that
$$\mathcal{A}_{F}(w_{1}) = \left( \sum_{e \in E_{X}} \alpha(e)\right) w_{1},$$
where $w_{1} = (v_{1},1_{G})$, from which it follows that for any $\psi \in \widehat{G}(F)$, one has
$$\mathcal{A}_{F}(e_{\psi}w_{1}) = \left(\sum_{e \in E_{X}}\psi(\alpha(e)) \right) e_{\psi}w_{1}.$$
Therefore, the $F$-basis $\{e_{\psi}w_{1}: \psi \in \widehat{G}(F) \}$ of $FV_{Y}$ consists of eigenvectors for $\mathcal{A}_{F}$.  The operator $\mathcal{A}_{F}$ is thus diagonalizable and in particular one has $\delta(Y) = 0$ as desired.
\end{proof}

%\section{Bowen--Franks groups and minus class groups}
\section{Bowen--Franks groups and minus class groups} \label{section:special_cover}
Throughout this section, let $\Delta = {\rm Gal}(K/\mathbb{Q}) \simeq \mathbb{F}_{p}^{\times}$, where $K = \mathbb{Q}(\zeta_{p})$ and $p$ is a fixed odd rational prime.  Let $j \in \Delta$ be the unique element of order $2$.  We write $\Delta^{+} = \Delta/\langle j \rangle$.  If $F$ is an algebraically closed field of characteristic zero, we will say that a character $\psi \in \widehat{\Delta}(F)$ is even if $\psi(j)=1$ and it is odd if $\psi(j) = -1$. 

%\subsection{Generalized Bernoulli numbers}
\subsection{Generalized Bernoulli numbers}

\begin{comment}
\DV{I need to modify the proposition below slightly to handle all rational primes $\ell \nmid p-1$.  I also need to define the Bernoulli numbers at some point.}
\begin{proposition} \label{bernoulli}
Let $i$ be an integer satisfying $1 \le i \le p-1$.  If $i$ is even, one has $B_{1,\omega^{-i}} = 0$.  If $i$ is odd, then $B_{1,\omega^{-i}} \neq 0$, and in particular
$$pB_{1,\omega^{-1}} \equiv p-1 \not \equiv 0 \pmod{p}.$$
\end{proposition}
\begin{proof}
\DV{Do that.  All in Washington.}
\end{proof}
\end{comment}
Consider the Stickelberger element
$$\theta = -\frac{1}{p}\sum_{i=1}^{p-1} i \sigma_{i}^{-1} \in \mathbb{Q}[\Delta]$$
from (\ref{stickelberger}).
\begin{definition}
Let $F$ be an algebraically closed field of characteristic zero, and let $\psi \in \widehat{\Delta}(F)$ be a non-trivial character.  The corresponding generalized Bernoulli $F$-number is defined to be
$$B_{1,\psi} = -\psi^{-1}(\theta) = \frac{1}{p} \sum_{i=1}^{p-1} i \psi(\sigma_{i}) \in F.$$
\end{definition}
We will make use of the following well-known result.
\begin{proposition} \label{bernoulli}
Let $F$ be an algebraically closed field of characteristic zero.  If $\psi \in \widehat{\Delta}(F)$ is a non-trivial even character, then $B_{1,\psi} = 0$, whereas if $\psi \in \widehat{\Delta}(F)$ is odd, then $B_{1,\psi} \neq 0$.  Moreover, if $\omega \in \widehat{\Delta}(\mathbb{C}_{p}) \subseteq \widehat{\Delta}(\mathbb{Z}_{p})$ is the Teichm\"{u}ller character, then
$$pB_{1,\omega^{-1}} \equiv p-1 \not \equiv 0 \pmod{p}.$$
\end{proposition}
\begin{proof}
The assertion on non-trivial even characters follows from the identity $\theta + j\theta = -N_{\Delta}$, where $N_{\Delta} = \sum_{\sigma \in \Delta} \sigma \in \mathbb{Z}[\Delta]$.  The assertion on odd characters follows from the class number formula on the minus side (\cite[Theorem 4.17]{Washington:1997}).  The final assertion on the Teichm\"{u}ller character is explained in  \cite[Page $101$]{Washington:1997}.
\end{proof}

%\subsection{A special Galois cover of digraphs}
\subsection{A special Galois cover of digraphs} \label{a_special_cover}
Our goal in this subsection is to construct a Galois cover $Y/X$ of digraphs with Galois group $\Delta$ for which
\begin{equation} \label{key}
\gamma_{Y/X}(1) = p \theta ,
\end{equation}
where $p\theta$ is the element in $\mathbb{Z}[\Delta]$ defined in (\ref{integralized_stick}).  We let
$$\delta = \frac{p-1}{2} p + 1 $$
and consider the bouquet digraph $X = B_{\delta}$ with $\delta$ loops.  Let us label the $\delta$ edges of $X$ as follows
$$E_{X} = \{e_{0},e_{1,1}, e_{2,1}, e_{2,2}, e_{3,1}, e_{3,2}, e_{3,3}, \ldots, e_{p-1,1},\ldots, e_{p-1,p-1}\}. $$
We define a voltage assignment $\alpha: E_{X} \rightarrow \Delta$ via   
$$e_{0} \mapsto \alpha(e_{0}) = \sigma_{1}^{-1} = \sigma_{1} \text{ and } e_{i,j} \mapsto \alpha(e_{i,j}) = \sigma_{i}^{-1}, $$
and we let $Y$ be the derived digraph $X(\Delta,\alpha)$.  \cref{connectedness} shows that $Y$ is strongly connected, and \cref{concrete_construction} implies that $Y/X$ is a Galois cover of digraphs with Galois group isomorphic to $\Delta$.  Then (\ref{what_we_need}) gives
\begin{equation} \label{stickelberger_from_digraph}
\gamma_{Y/X}(1) = - \sum_{i=1}^{p-1}i\sigma_{i}^{-1} = p \theta \in \mathbb{Z}[\Delta],
\end{equation}
so the Galois cover $Y/X$ satisfies (\ref{key}).  Moreover, \cref{delta_zero_cs} shows that $\delta(Y) = 0$ (and clearly $\delta(X) = 0$ as well).

For instance, for $p=3$, the Galois cover $X(\Delta,\alpha) \rightarrow X$ is the $2$-to-$1$ cover given by
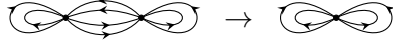
\begin{equation*}
\begin{tikzpicture} [baseline={([yshift=-0.7ex] current bounding box.center)}]
%vertices
\draw[fill=black] (0,0) circle (1pt);
\draw[fill=black] (1,0) circle (1pt);

%edges
\draw (0,0) edge [decoration={markings, mark= at position 0.44 with {\arrow[xscale=-1]{stealth}}},preaction={decorate},bend left=20] (1,0);
\draw (0,0) edge [decoration={markings, mark= at position 0.58 with {\arrow{stealth}}},preaction={decorate}, bend right=20] (1,0);
\draw (0,0) edge [decoration={markings, mark= at position 0.44 with {\arrow[xscale=-1]{stealth}}},preaction={decorate},bend left=50] (1,0);
\draw (0,0) edge [decoration={markings, mark= at position 0.58 with {\arrow{stealth}}},preaction={decorate}, bend right=50] (1,0);

\draw (0,0) edge [decoration={markings, mark= at position 0.38 with {\arrow{stealth}}},preaction = {decorate}, loop left, in = 155, out = 205,min distance=8mm] (0,0);
\draw (1,0) edge [decoration={markings, mark= at position 0.38 with {\arrow{stealth}}},preaction = {decorate}, loop right, in = 25, out = 335,min distance=8mm] (1,0);
\draw (0,0) edge [decoration={markings, mark= at position 0.60 with {\arrow{stealth}}},preaction = {decorate}, loop left, in = 145, out = 215,min distance=12mm] (0,0);
\draw (1,0) edge [decoration={markings, mark= at position 0.60 with {\arrow{stealth}}},preaction = {decorate}, loop right, in = 35, out = 325,min distance=12mm] (1,0);
\end{tikzpicture} 
\rightarrow
\begin{tikzpicture} [baseline={([yshift=-0.7ex] current bounding box.center)}]
%vertices
\draw[fill=black] (0,0) circle (1pt);

%edges
\draw (0,0) edge [decoration={markings, mark= at position 0.38 with {\arrow{stealth}}},preaction = {decorate}, loop left, in = 155, out = 205,min distance=8mm] (0,0);
\draw (0,0) edge [decoration={markings, mark= at position 0.60 with {\arrow{stealth}}},preaction = {decorate}, loop left, in = 145, out = 215,min distance=12mm] (0,0);
\draw (0,0) edge [decoration={markings, mark= at position 0.38 with {\arrow{stealth}}},preaction = {decorate}, loop right, in = 25, out = 335,min distance=8mm] (0,0);
\draw (0,0) edge [decoration={markings, mark= at position 0.60 with {\arrow{stealth}}},preaction = {decorate}, loop right, in = 35, out = 325,min distance=12mm] (0,0);

\end{tikzpicture} 
\end{equation*}

For $p=5$, one obtains a digraph $X(\Delta,\alpha)$ that is a Galois $4$-to-$1$ cover of a bouquet digraph with $11$ loops and whose Galois group is cyclic of order four.  The digraph $X(\Delta,\alpha)$ is given by  
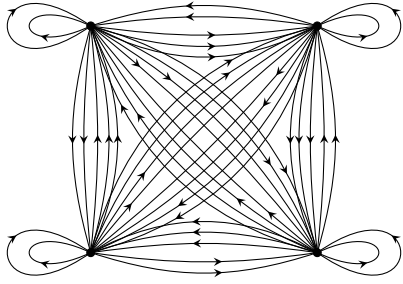
\begin{equation*}
\begin{tikzpicture}[scale=1.5]
%vertices
\draw[fill=black] (0,0) circle (1pt);
\draw[fill=black] (2,0) circle (1pt);
\draw[fill=black] (0,2) circle (1pt);
\draw[fill=black] (2,2) circle (1pt);

%edges
\draw (0,0) edge [decoration={markings, mark= at position 0.38 with {\arrow{stealth}}},preaction = {decorate}, loop left, in = 155, out = 205,min distance=8mm] (0,0);
\draw (0,0) edge [decoration={markings, mark= at position 0.60 with {\arrow{stealth}}},preaction = {decorate}, loop left, in = 145, out = 215,min distance=12mm] (0,0);
\draw (0,2) edge [decoration={markings, mark= at position 0.38 with {\arrow{stealth}}},preaction = {decorate}, loop left, in = 155, out = 205,min distance=8mm] (0,2);
\draw (0,2) edge [decoration={markings, mark= at position 0.60 with {\arrow{stealth}}},preaction = {decorate}, loop left, in = 145, out = 215,min distance=12mm] (0,2);

\draw (2,0) edge [decoration={markings, mark= at position 0.38 with {\arrow{stealth}}},preaction = {decorate}, loop right, in = 25, out = 335,min distance=8mm] (2,0);
\draw (2,2) edge [decoration={markings, mark= at position 0.38 with {\arrow{stealth}}},preaction = {decorate}, loop right, in = 25, out = 335,min distance=8mm] (2,2);
\draw (2,0) edge [decoration={markings, mark= at position 0.60 with {\arrow{stealth}}},preaction = {decorate}, loop right, in = 35, out = 325,min distance=12mm] (2,0);
\draw (2,2) edge [decoration={markings, mark= at position 0.60 with {\arrow{stealth}}},preaction = {decorate}, loop right, in = 35, out = 325,min distance=12mm] (2,2);

%blue ones
\draw (2,2) edge [decoration={markings, mark= at position 0.58 with {\arrow{stealth}}},preaction={decorate}, bend right=8] (0,2);
\draw (0,2) edge [decoration={markings, mark= at position 0.52 with {\arrow{stealth}}},preaction={decorate}, bend right=6] (0,0);
\draw (0,0) edge [decoration={markings, mark= at position 0.58 with {\arrow{stealth}}},preaction={decorate}, bend right=8] (2,0);
\draw (2,0) edge [decoration={markings, mark= at position 0.52 with {\arrow{stealth}}},preaction={decorate}, bend right=6] (2,2);

\draw (2,2) edge [decoration={markings, mark= at position 0.58 with {\arrow{stealth}}},preaction={decorate}, bend right=18] (0,2);
\draw (0,2) edge [decoration={markings, mark= at position 0.52 with {\arrow{stealth}}},preaction={decorate}, bend right=16] (0,0);
\draw (0,0) edge [decoration={markings, mark= at position 0.58 with {\arrow{stealth}}},preaction={decorate}, bend right=18] (2,0);
\draw (2,0) edge [decoration={markings, mark= at position 0.52 with {\arrow{stealth}}},preaction={decorate}, bend right=16] (2,2);

%red ones

\draw (0,2) edge [decoration={markings, mark= at position 0.55 with {\arrow{stealth}}},preaction={decorate}, bend right=8] (2,2);
\draw (0,0) edge [decoration={markings, mark= at position 0.52 with {\arrow{stealth}}},preaction={decorate}, bend right=6] (0,2);
\draw (2,2) edge [decoration={markings, mark= at position 0.52 with {\arrow{stealth}}},preaction={decorate}, bend right=6] (2,0);
\draw (2,0) edge [decoration={markings, mark= at position 0.55 with {\arrow{stealth}}},preaction={decorate}, bend right=8] (0,0);

\draw (0,2) edge [decoration={markings, mark= at position 0.55 with {\arrow{stealth}}},preaction={decorate}, bend right=18] (2,2);
\draw (0,0) edge [decoration={markings, mark= at position 0.52 with {\arrow{stealth}}},preaction={decorate}, bend right=16] (0,2);
\draw (2,2) edge [decoration={markings, mark= at position 0.52 with {\arrow{stealth}}},preaction={decorate}, bend right=16] (2,0);
\draw (2,0) edge [decoration={markings, mark= at position 0.55 with {\arrow{stealth}}},preaction={decorate}, bend right=18] (0,0);

\draw (0,2) edge [decoration={markings, mark= at position 0.55 with {\arrow{stealth}}},preaction={decorate}, bend right=28] (2,2);
\draw (0,0) edge [decoration={markings, mark= at position 0.52 with {\arrow{stealth}}},preaction={decorate}, bend right=24] (0,2);
\draw (2,2) edge [decoration={markings, mark= at position 0.52 with {\arrow{stealth}}},preaction={decorate}, bend right=24] (2,0);
\draw (2,0) edge [decoration={markings, mark= at position 0.55 with {\arrow{stealth}}},preaction={decorate}, bend right=28] (0,0);

%green ones

\draw (2,0) edge [decoration={markings, mark= at position 0.20 with {\arrow{stealth}}},preaction={decorate}, bend left=6] (0,2);
\draw (0,2) edge [decoration={markings, mark= at position 0.20 with {\arrow{stealth}}},preaction={decorate}, bend left=6] (2,0);
\draw (0,0) edge [decoration={markings, mark= at position 0.20 with {\arrow{stealth}}},preaction={decorate}, bend left=6] (2,2);
\draw (2,2) edge [decoration={markings, mark= at position 0.20 with {\arrow{stealth}}},preaction={decorate}, bend left=6] (0,0);

\draw (2,0) edge [decoration={markings, mark= at position 0.30 with {\arrow{stealth}}},preaction={decorate}, bend left=14] (0,2);
\draw (0,2) edge [decoration={markings, mark= at position 0.30 with {\arrow{stealth}}},preaction={decorate}, bend left=14] (2,0);
\draw (0,0) edge [decoration={markings, mark= at position 0.30 with {\arrow{stealth}}},preaction={decorate}, bend left=14] (2,2);
\draw (2,2) edge [decoration={markings, mark= at position 0.30 with {\arrow{stealth}}},preaction={decorate}, bend left=14] (0,0);

\draw (2,0) edge [decoration={markings, mark= at position 0.70 with {\arrow{stealth}}},preaction={decorate}, bend left=22] (0,2);
\draw (0,2) edge [decoration={markings, mark= at position 0.70 with {\arrow{stealth}}},preaction={decorate}, bend left=22] (2,0);
\draw (0,0) edge [decoration={markings, mark= at position 0.70 with {\arrow{stealth}}},preaction={decorate}, bend left=22] (2,2);
\draw (2,2) edge [decoration={markings, mark= at position 0.70 with {\arrow{stealth}}},preaction={decorate}, bend left=22] (0,0);

\draw (2,0) edge [decoration={markings, mark= at position 0.75 with {\arrow{stealth}}},preaction={decorate}, bend left=30] (0,2);
\draw (0,2) edge [decoration={markings, mark= at position 0.75 with {\arrow{stealth}}},preaction={decorate}, bend left=30] (2,0);
\draw (0,0) edge [decoration={markings, mark= at position 0.73 with {\arrow{stealth}}},preaction={decorate}, bend left=30] (2,2);
\draw (2,2) edge [decoration={markings, mark= at position 0.73 with {\arrow{stealth}}},preaction={decorate}, bend left=30] (0,0);

\end{tikzpicture} 
\end{equation*}

%\subsection{Cardinalities Bowen--Franks groups and minus class groups}
\subsection{Cardinalities of Bowen--Franks groups and minus class groups} \label{bf_and_cg}
Throughout this section, we let $Y/X$ be the Galois cover from \cref{a_special_cover} with Galois group $\Delta \simeq \mathbb{F}_{p}^{\times}$.  We also consider the intermediate cover $Y^{+} = Y_{\langle j \rangle}$ given by \cref{gal_cor}.  Our goal in this subsection is to formulate and prove \cref{goal} below.  Our starting point is as follows. \cref{decomp} implies, on the one hand
$$g_{Y}^{*}(1) = \prod_{\psi \in \widehat{\Delta}(\mathbb{C})} g_{Y/X}^{*}(1,\psi),$$
and, on the other hand
$$g_{Y^{+}}^{*}(1) = \prod_{\psi \in \widehat{\Delta^{+}}(\mathbb{C})} g_{Y^{+}/X}^{*}(1,\psi).$$
Now, \cref{inflation} implies the equality
$$\prod_{\psi \in \widehat{\Delta^{+}}(\mathbb{C})} g_{Y^{+}/X}^{*}(1,\psi) = \prod_{\substack{\psi \in \widehat{\Delta}(\mathbb{C}) \\ \psi \text{ even }}} g_{Y/X}^{*}(1,\psi),$$
from which it follows that
\begin{equation} \label{key_identity}
\frac{g_{Y}^{*}(1)}{g_{Y^{+}}^{*}(1)} = \prod_{\substack{\psi \in \widehat{\Delta}(\mathbb{C}) \\ \psi \text{ odd }}} g_{Y/X}^{*}(1,\psi).
\end{equation}
We shall first determine $m(Y^{+})$ and $m(Y)$ from \cref{the_special_value}.  Before doing so, we first calculate $\gamma_{Y^{+}/X}(1)$ explicitly in the following proposition.

\begin{proposition} \label{equi_spe_plus}
With the notation as above, one has
$$\gamma_{Y^{+}/X}(1) = -p \cdot N_{\Delta^{+}} \in \mathbb{Z}[\Delta^{+}],$$
where $N_{\Delta^{+}} = \sum_{\sigma \in \Delta^{+}} \sigma \in \mathbb{Z}[\Delta^{+}]$.
\end{proposition}
\begin{proof}
Let $\pi:\mathbb{Z}[\Delta] \rightarrow \mathbb{Z}[\Delta^{+}]$ be the natural projection map.  \cref{equiv_inf} shows that
$$\gamma_{Y^{+}/X}(1) = \pi\left(\gamma_{Y/X}(1) \right) = \pi(p \theta).$$
Now, one has
\begin{equation*}
\begin{aligned}
\pi(p \theta) &= - \sum_{i=1}^{\frac{p-1}{2}} i \pi(\sigma_{i}^{-1}) - \sum_{i = \frac{p+1}{2}}^{p-1}i \pi(\sigma_{i}^{-1}) \\
&= - \left(\sum_{i=1}^{\frac{p-1}{2}} i \pi(\sigma_{i}^{-1}) + \sum_{i=\frac{p+1}{2}}^{p-1} i \pi(j \sigma_{-i}^{-1}) \right) \\
&= - \left(\sum_{i=1}^{\frac{p-1}{2}} i \pi(\sigma_{i}^{-1}) + \sum_{i=1}^{\frac{p-1}{2}} (p-i) \pi(j \sigma_{i-p}^{-1})\right) \\
&= -p N_{\Delta^{+}},
\end{aligned}
\end{equation*}
as desired. 
\end{proof}

\cref{equi_spe_plus} allows us to describe the structure of ${\rm BF}(Y^{+})$ as a $\mathbb{Z}$-module. 
\begin{proposition} \label{bf_plus}
With the notation as above, one has an isomorphism
$${\rm BF}(Y^{+}) \simeq \mathbb{Z}^{\frac{p-3}{2}} \oplus \mathbb{Z}/p\mathbb{Z} $$
of $\mathbb{Z}$-modules.
\end{proposition}
\begin{proof}
We start with the finite free presentation
$$\mathbb{Z}V_{Y^{+}} \stackrel{\mathcal{BF}_{Y^{+}}}{\longrightarrow} \mathbb{Z}V_{Y^{+}} \rightarrow {\rm BF}(Y^{+}) \rightarrow 0$$
of ${\rm BF}(Y^{+})$ as $\mathbb{Z}[\Delta^{+}]$-modules given by (\ref{finite_free_presentation}).  Note that in our current setting, $\mathbb{Z}V_{Y^{+}}$ is a free $\mathbb{Z}[\Delta^{+}]$-module of rank $1$.  It follows that 
$$\mathcal{BF}_{Y^{+}}(D) = \gamma_{Y^{+}/X}(1) \cdot D $$
for all $D \in \mathbb{Z}V_{Y^{+}}$.  Therefore,
$${\rm BF}(Y^{+}) \simeq \mathbb{Z}[\Delta^{+}]/(\gamma_{Y^{+}/X}(1)),$$
as $\mathbb{Z}[\Delta^{+}]$-modules.  Using \cref{equi_spe_plus}, we consider now the short exact sequence
\begin{equation} \label{imp_ses}
0 \rightarrow N_{\Delta^{+}} \mathbb{Z}[\Delta^{+}]/(\gamma_{Y^{+}/X}(1)) \rightarrow \mathbb{Z}[\Delta^{+}]/(\gamma_{Y^{+}/X}(1)) \rightarrow \mathbb{Z}[\Delta^{+}]/N_{\Delta^{+}}\mathbb{Z}[\Delta^{+}] \rightarrow 0
\end{equation}
of $\mathbb{Z}[\Delta^{+}]$-modules.  As $\mathbb{Z}$-modules, we have 
$$N_{\Delta^{+}} \mathbb{Z}[\Delta^{+}]/(\gamma_{Y^{+}/X}(1)) \simeq \mathbb{Z}/p\mathbb{Z},$$
because of \cref{equi_spe_plus}.  Moreover, the term on the right of the short exact sequence (\ref{imp_ses}) does not have $\mathbb{Z}$-torsion, and thus is a free $\mathbb{Z}$-module of finite rank $(p-3)/2$.  It follows that the short exact sequence (\ref{imp_ses}) splits in the category of $\mathbb{Z}$-modules, and the result follows.
\end{proof}

We can now understand the special value $g_{Y^{+}}^{*}(1)$ completely.

\begin{corollary} \label{plus_part}
With the notation as above, one has
$$m(Y^{+}) = (-1)^{\frac{p-1}{2}} \cdot \frac{p-1}{2} \text{ and } \#{\rm BF}(Y^{+})_{\rm tors} = p,$$
and
$$g_{Y^{+}}^{*}(1) = (-1)^{\frac{p-1}{2}} \cdot \frac{p-1}{2} \cdot p.$$
\end{corollary}
\begin{proof}
The eigenvalues of the operator $\mathcal{BF}_{\mathbb{C}}:\mathbb{C}V_{Y^{+}} \rightarrow \mathbb{C}V_{Y^{+}}$ are given by $\psi(\gamma_{Y^{+}/X}(1))$ as $\psi$ runs over the $\mathbb{C}$-valued characters of $\Delta^{+}$.  Therefore, \cref{equi_spe_plus} implies that $0$ is an eigenvalue of $\mathcal{BF}_{\mathbb{C}}$ with algebraic multiplicity $(p-3)/2$ and that the eigenvalue
$$\psi_{0} \left( \gamma_{Y^{+}/X}(1) \right) = - \frac{p-1}{2} p$$
has algebraic multiplicity $1$.  Now, \cref{zeta_explicit} implies that
$$g_{Y^{+}}(u) = (1-u)^{\frac{p-3}{2}} \cdot \left(1 - \left( 1+ \frac{p-1}{2}p \right)u \right),$$ 
and also
$$g_{Y^{+}}^{*}(1) = (-1)^{\frac{p-1}{2}} \cdot \frac{p-1}{2} \cdot p. $$
Combining with \cref{identification,delta_zero_cs}, \cref{the_special_value}, and \cref{bf_plus} gives the desired result. 
\end{proof}

For the digraph $Y$, the integer $m(Y)$ from \cref{the_special_value} can now be calculated explicitly.
\begin{proposition}\label{thm:m(Y)}
With the notation as above, one has
    \[|m(Y)|= 2^{\frac{p-3}{2}} \cdot\frac{p-1}{2}.\]
\end{proposition}

\begin{proof}
We have a natural commutative diagram
\[
\begin{tikzcd}
0\arrow[r]&(1-j)\mathbb{Z}V_{Y}\arrow[r]\arrow[d,"\mathcal{BF}_{Y}^{-}"]&\mathbb{Z}V_{Y}\arrow[r,"f"]\arrow[d,"\mathcal{BF}_{Y}"]&\mathbb{Z}V_{Y^{+}}\arrow[r]\arrow[d,"\mathcal{BF}_{Y^{+}}"]&0\\
0\arrow[r]&(1-j)\mathbb{Z}V_{Y}\arrow[r]&\mathbb{Z}V_{Y}\arrow[r,"f"]&\mathbb{Z}V_{Y^{+}}\arrow[r]&0,
\end{tikzcd}
\]
where $\mathcal{BF}_{Y}^{-}$ is the restriction of the Bowen--Franks operator $\mathcal{BF}_{Y}$ to $(1-j)\mathbb{Z}V_{Y}$, and the two right horizontal maps are the natural projection map $f:\mathbb{Z}V_{Y} \rightarrow \mathbb{Z}V_{Y^{+}}$.  Note that since $\mathbb{Z}V_{Y}$ is a free $\mathbb{Z}[\Delta]$-module of rank one, the Bowen--Franks operator $\mathcal{BF}_{Y}$ is simply given by
\begin{equation}
\label{eq}
    \mathcal{BF}_{Y}(D) = \gamma_{Y/X}(1)\cdot D = p\theta \cdot D,
\end{equation}
for all $D \in \mathbb{Z}V_{Y}$.  

For what follows, we view $\mathbb{Z}V_{Y} \subseteq \mathbb{C}V_{Y}$. By \eqref{eq}, the eigenvalues of $\mathcal{BF}_{Y}^{-}$ are given by $\psi(p\theta)$ for odd characters $\psi$. By \cref{bernoulli}, zero is not an eigenvalue of  $\mathcal{BF}_{Y}^{-}$.  %If $D \in {\rm ker}(\mathcal{BF}_{Y}^{-}) = {\rm ker}(\mathcal{BF}_{Y}) \cap (1-j)\mathbb{Z}V_{Y}$, then for an odd character $\psi \in \widehat{\Delta}(\mathbb{C})$, we have $0 = e_{\psi} p\theta  D = p \psi(\theta) e_{\psi} D$,
%and \cref{bernoulli}  implies that $e_{\psi}D = 0$, and for an even character $\psi \in \widehat{\Delta}(\mathbb{C})$, since $jD = -D$, we get $2e_{\psi}D = 0$, from which it follows that $e_{\psi}D=0$ as well.
As $(1-j)\Z V_Y$ is a free $\Z$-module, $\mathcal{BF}_{Y}^{-}$ is injective. 
The snake lemma gives an exact sequence
$$0 \rightarrow {\rm ker}(\mathcal{BF}_{Y^{+}})/f({\rm ker}(\mathcal{BF}_{Y})) \rightarrow {\rm coker}(\mathcal{BF}_{Y}^{-}) \rightarrow {\rm BF}(Y) \rightarrow {\rm BF}(Y^{+}) \rightarrow 0. $$
Since $r_{Y}=r_{Y^{+}} = (p-3)/2$, \cref{divi_pro} gives in fact an exact sequence
\begin{equation} \label{unn}
0 \rightarrow {\rm ker}(\mathcal{BF}_{Y^{+}})/f({\rm ker}(\mathcal{BF}_{Y})) \rightarrow {\rm coker}(\mathcal{BF}_{Y}^{-}) \rightarrow {\rm BF}(Y)_{\rm tors} \rightarrow {\rm BF}(Y^{+})_{\rm tors} \rightarrow 0 
\end{equation}
of finite abelian groups, and we will now determine the cardinality of the first two.

To determine the cardinality of ${\rm ker}(\mathcal{BF}_{Y^{+}})/f({\rm ker}(\mathcal{BF}_{Y}))$, label the vertices of $Y^{+}$, say $V_{Y^{+}} = \{v_{1},\ldots,v_{\frac{p-1}{2}} \}$, and for each $i = 1,\ldots (p-1)/2$, fix one of the vertices of $Y$ lying above $v_{i}$ and denote it by $w_{i}$.  Consider now the $(p-3)/2$ elements
$$D_{i} = w_{i} + jw_{i} - (w_{1} + jw_{1}) \in \mathbb{Z}V_{Y}$$
for $i = 2,\ldots,(p-1)/2$.  Note that since $\Delta$ acts transitively on the vertices of $Y$, there exists $\sigma_{i} \in \Delta$ such that $\sigma_{i}w_{1} = w_{i}$.  Thus, we have
$$D_{i} = (\sigma_{i}-1)(1+j)w_{1}$$
for $i=2,\ldots,(p-1)/2$.  A character-by-character argument similar to the one used to show the injectivity of $\mathcal{BF}_{Y}^{-}$ above shows that $D_{i} \in {\rm ker}(\mathcal{BF}_{Y})$ for all $i=2,\ldots,(p-1)/2$.  One verifies as well that the $D_{i}$ are $\mathbb{Z}$-linearly independent.  Consider the two free $\mathbb{Z}$-submodules of $\mathbb{Z}V_{Y}$ given by
$$M = \bigoplus_{i=2}^{\frac{p-1}{2}}\mathbb{Z}D_{i} \text{ and } N = \bigoplus_{i=2}^{\frac{p-1}{2}}\mathbb{Z}w_{i}.$$
The natural projection map $\pi:\mathbb{Z}V_{Y} \rightarrow N$ admits as a splitting the group morphism $\iota:N \rightarrow \mathbb{Z}V_{Y}$ given by $w_{i} \mapsto \iota(w_{i}) = D_{i}$.  Since $\iota(N) = M$, one gets a direct sum decomposition $\mathbb{Z}V_{Y} = {\rm ker}(\pi) \oplus M$, from which it follows that $M$ is saturated in $\mathbb{Z}V_{Y}$.  Since $M \subseteq {\rm ker}(\mathcal{BF}_{Y})$ and both $M$ and ${\rm ker}(\mathcal{BF}_{Y})$ have the same rank $(p-3)/2$ over $\mathbb{Z}$, it follows that $M = {\rm ker}(\mathcal{BF}_{Y})$.  Now, note that $f(D_{i}) = 2(v_{i} - v_{1})$, and that the elements $v_{i}-v_{1}$ form a $\mathbb{Z}$-basis for ${\rm ker}(\mathcal{BF}_{Y^{+}})$ by \cref{equi_spe_plus}.  Therefore, $f({\rm ker}(\mathcal{BF}_{Y})) = 2 {\rm ker}(\mathcal{BF}_{Y^{+}})$, and we obtain
\begin{equation} \label{deuxx}
[{\rm ker}(\mathcal{BF}_{Y^{+}}):f({\rm ker}(\mathcal{BF}_{Y})] = 2^{\frac{p-3}{2}}.
\end{equation}

Since $\mathcal{BF}_{Y}^{-}$ is injective, the cardinality of ${\rm coker}(\mathcal{BF}_{Y}^{-})$ is given by the absolute value of its determinant, which after a base change to $\mathbb{C}$ is given by
\begin{equation} \label{troiss}
\# {\rm coker}(\mathcal{BF}_{Y}^{-}) = \Big|\prod_{\substack{\psi \in \widehat{\Delta}(\mathbb{C}) \\ \psi \text{ odd }}} \psi(p\theta) \Big|.
\end{equation}
Putting (\ref{unn}), (\ref{deuxx}), (\ref{troiss}) together and using \cref{spe_decomp}, one gets
$$2^{p-3} \frac{\# {\rm BF}(Y)_{\rm tors}}{\# {\rm BF}(Y^{+})_{\rm tors}} =  \Big|\prod_{\substack{\psi \in \widehat{\Delta}(\mathbb{C}) \\ \psi \text{ odd }}} \psi(p\theta) \Big| = \Big|\frac{g_{Y}^{*}(1)}{g_{Y^{+}}^{*}(1)} \Big| = \frac{|m(Y)|\# {\rm BF}(Y)_{\rm tors}}{|m(Y^{+})|\# {\rm BF}(Y^{+})_{\rm tors}}.$$
Combined with \cref{plus_part}, it follows that
$$|m(Y)| = 2^{p-3} \cdot \frac{p-1}{2}.$$

\end{proof}

We are now in a position to formulate and prove our main result of this subsection.
\begin{theorem} \label{goal}
Let $Y/X$ be the Galois cover from \cref{a_special_cover} with Galois group $\Delta$.  Then, one has
$$\#{\rm BF}(Y)_{\rm tors} = p^{\frac{p-1}{2}} \cdot h^{-}(K),$$
where $h^{-}(K) = \# {\rm Cl}^{-}(K)$ is the minus class number of $K$.
\end{theorem}
\begin{proof}
Starting with (\ref{key_identity}), we have
\begin{equation} \label{gt}
\frac{g_{Y}^{*}(1)}{g_{Y^{+}}^{*}(1)} = p^{\frac{p-1}{2}} \prod_{\substack{\psi \in \widehat{\Delta}(\mathbb{C}) \\ \psi \text{ odd }}} (-B_{1,\psi^{-1}}). 
\end{equation}
The classical minus class number formula (\cite[Theorem 4.17]{Washington:1997} and \cite[Corollary 4.13]{Washington:1997} and see also \cite{Hasse:2019}) gives
$$h^{-}(K) = 2p \prod_{\substack{\psi \in \widehat{\Delta}(\mathbb{C}) \\ \psi \text{ odd }}}\left(-\frac{1}{2} B_{1,\psi^{-1}} \right),$$
and thus
\begin{equation} \label{nt}
2^{\frac{p-3}{2}} \cdot h^{-}(K) =p \prod_{\substack{\psi \in \widehat{\Delta}(\mathbb{C}) \\ \psi \text{ odd }}}\left(- B_{1,\psi^{-1}} \right).
\end{equation}
Putting (\ref{gt}) and (\ref{nt}) together give
$$\frac{g_{Y}^{*}(1)}{g_{Y^{+}}^{*}(1)} = p^{\frac{p-3}{2}} \cdot 2^{\frac{p-3}{2}} \cdot h^{-}(K).$$
Using \cref{plus_part} and \cref{thm:m(Y),delta_zero_cs}, one gets
$$\#{\rm BF}(Y)_{\rm tors} = p^{\frac{p-1}{2}} \cdot h^{-}(K),$$
as desired.
\end{proof}

\begin{remark}
It follows from the proof of \cref{goal} that
$$m(Y) = (-1)^{\frac{p-1}{2}} \cdot 2^{\frac{p-3}{2}} \cdot \frac{p-1}{2}.$$
\end{remark}

%\subsection{The \texorpdfstring{$\ell$}{}-part of the Bowen--Franks group when \texorpdfstring{$\ell \nmid \# \Delta$}{}}
\subsection{The \texorpdfstring{$\ell$}{}-part of the Bowen--Franks group when \texorpdfstring{$\ell \nmid \# \Delta$}{}} \label{local_analysis}
Recall that if $M$ is a finite abelian group, and $\ell$ is a rational prime, we will write $M_{\ell}$ instead of $M[\ell^{\infty}]$.  In particular, we will write
${\rm BF}_{\ell}(Y)_{\rm tors}$ and ${\rm Cl}^{-}_{\ell}(K)$ instead of ${\rm BF}(Y)_{\rm tors}[\ell^{\infty}]$ and ${\rm Cl}^{-}(K)[\ell^{\infty}]$, respectively.  More generally, if $M$ is a finitely generated abelian group, we write $M_{\ell}$ instead of $\mathbb{Z}_{\ell} \otimes_{\mathbb{Z}} M$; for instance, ${\rm BF}_{\ell}(Y)$.

Throughout this subsection, we let $Y/X$ be the Galois cover of digraphs with Galois group $\Delta$ constructed in \cref{a_special_cover}.  By taking the $p$-part of the equality of \cref{goal}, one has
\begin{equation} \label{equality_at_p}
\#{\rm BF}_{p}(Y)_{\rm tors} = p^{\frac{p-1}{2}} \cdot \#{\rm Cl}^{-}_{p}(K).
\end{equation}
Similarly, if $\ell$ is a rational prime different than $p$, then
\begin{equation} \label{equality_at_q}
\#{\rm BF}_{\ell}(Y)_{\rm tors} = \#{\rm Cl}^{-}_{\ell}(K).
\end{equation}
Let now $\ell$ be a rational prime satisfying 
\begin{equation} \label{ss_hyp}
\ell \nmid \# \Delta = p-1.
\end{equation}

As before, let $j = \sigma_{-1} \in \Delta$ be the usual complex conjugation.  Consider now the idempotents
$$e^{+} = \frac{1+j}{2}, e^{-} = \frac{1-j}{2} \in \mathbb{Z}_{\ell}[\Delta].$$
If $M$ is any $\mathbb{Z}_{\ell}[\Delta]$ module, then we define
$$M^{+} = e^{+} M \text{ and } M^{-} = e^{-} M.$$

\begin{proposition} \label{minus_vs_torsion}
With the notation as above, and for a rational prime $\ell$ satisfying (\ref{ss_hyp}), the group ${\rm BF}_{\ell}^{-}(Y)$ is finite.  If $\ell \neq p$, then 
$${\rm BF}_{\ell}(Y)_{\rm tors} = {\rm BF}_{\ell}^{-}(Y),$$
and otherwise
$${\rm BF}_{p}(Y)_{\rm tors} = {\rm BF}_{p}^{-}(Y) \oplus e_{\psi_{0}} {\rm BF}_{p}(Y).$$
Moreover, $e_{\psi_{0}} {\rm BF}_{p}(Y) \simeq \mathbb{Z}/p\mathbb{Z}$.
\end{proposition}
\begin{proof}
Starting with the finite free presentation (\ref{finite_free_presentation}), identifying $\mathbb{Z}V_{Y}$ with $\mathbb{Z}[\Delta]$, and tensoring with $\mathbb{Z}_{\ell}$ over $\mathbb{Z}$ gives another finite free presentation
$$\mathbb{Z}_{\ell}[\Delta] \stackrel{m}{\longrightarrow} \mathbb{Z}_{\ell}[\Delta] \rightarrow {\rm BF}_{\ell}(Y) \rightarrow 0,$$
where $m$ is the multiplication by $\gamma_{Y/X}(1)$ map.  \cref{bernoulli} implies that the image of $m$ on the minus side, namely $e^{-}\gamma_{Y/X}(1) \mathbb{Z}_{\ell}[\Delta]$, has rank $(p-1)/2$ over $\mathbb{Z}_{\ell}$ and since $\psi_{0}(p\theta) \in \mathbb{Z}_{\ell}^{\times}$ if $\ell \neq p$, we get in this case
$${\rm BF}_{\ell}(Y)_{\rm tors} = {\rm BF}_{\ell}^{-}(Y).$$
On the other hand, if $\ell = p$, then $\psi_{0}(p\theta) \in p\mathbb{Z}_{p}$, so we get
$${\rm BF}_{p}(Y)_{\rm tors} = {\rm BF}_{p}^{-}(Y) \oplus e_{\psi_{0}}{\rm BF}_{p}(Y),$$
and $e_{\psi_{0}}{\rm BF}_{p}(Y) \simeq \mathbb{Z}/p\mathbb{Z}$.
\end{proof}

Combining with both (\ref{equality_at_p}) and (\ref{equality_at_q}), one obtains the following result.
\begin{corollary}
With the same notation as above, if $\ell$ is a rational prime satisfying (\ref{ss_hyp}) such that $\ell \neq p$, then
$$\#{\rm BF}_{\ell}^{-}(Y) =  \#{\rm Cl}_{\ell}^{-}(K) \text{ and } \#{\rm BF}_{p}^{-}(Y) = p^{\frac{p-3}{2}} \cdot \#{\rm Cl}_{p}^{-}(K). $$
\end{corollary}

Let now $\mathcal{O}$ be the valuation ring of the local field $\mathbb{Q}_{\ell}(\mu_{p-1})$.  Given $\psi \in \widehat{\Delta}(\mathbb{C}_{\ell}) \subseteq \widehat{\Delta}(\mathcal{O})$, one has the usual associated idempotent
$$e_{\psi} = \frac{1}{p-1}\sum_{\sigma \in \Delta}\psi(\sigma) \sigma^{-1} \in \mathcal{O}[\Delta].$$
Our goal now is to compare the isotypic components corresponding to odd characters of both $\mathcal{O}[\Delta]$-modules 
$${\rm BF}_{\mathcal{O}}(Y):= \mathcal{O} \otimes_{\mathbb{Z}_{\ell}} {\rm BF}_{\ell}(Y) \text{ and } {\rm Cl}^{-}_{\mathcal{O}}(K):=\mathcal{O} \otimes_{\mathbb{Z}_{\ell}} {\rm Cl}^{-}_{\ell}(K),$$
provided $\ell$ is a rational prime satisfying (\ref{ss_hyp}).

\begin{theorem} \label{key_iso1}
Consider the Galois cover $Y/X$ defined in \cref{a_special_cover}, and let $\ell$ be a rational prime satisfying (\ref{ss_hyp}).  With the notation as above, if $\psi \in \widehat{\Delta}(\mathbb{C}_{\ell}) \subseteq \widehat{\Delta}(\mathcal{O})$, then
$$e_{\psi} \cdot {\rm BF}_{\mathcal{O}}(Y) \simeq \mathcal{O}/\psi(p\theta)\mathcal{O},$$
so that in particular for an even character $\psi \in \widehat{\Delta}(\mathbb{C}_{\ell}) \subseteq \widehat{\Delta}(\mathcal{O})$ different from the trivial character, one has
$$e_{\psi} \cdot {\rm BF}_{\mathcal{O}}(Y) \simeq \mathcal{O}.$$
\end{theorem}
\begin{proof}
Starting with the finite free presentation (\ref{finite_free_presentation}), identifying $\mathbb{Z}V_{Y}$ with $\mathbb{Z}[\Delta]$, and tensoring with $\mathcal{O}$ over $\mathbb{Z}$ gives another finite free presentation
\begin{equation} \label{simple_ffp1}
\mathcal{O}[\Delta] \stackrel{\mathcal{BF}}{\longrightarrow} \mathcal{O}[\Delta] \rightarrow {\rm BF}_{\mathcal{O}}(Y) \rightarrow 0,
\end{equation}
but this time over $\mathcal{O}[\Delta]$.  Given $\psi \in \widehat{\Delta}(\mathbb{C}_{\ell})$ and taking the corresponding isotypic component of (\ref{simple_ffp1}) gives an exact sequence
$$e_{\psi} \mathcal{O}[\Delta] \stackrel{\mathcal{BF}^{\psi}}{\longrightarrow} e_{\psi} \mathcal{O}[\Delta] \rightarrow e_{\psi}  {\rm BF}_{\mathcal{O}}(Y) \rightarrow 0,$$
now of $\mathcal{O}$-modules.  Noticing that $\psi(p \cdot \theta) = -p \cdot B_{1,\psi^{-1}}$ when $\psi \neq \psi_{0}$, it suffices to combine with \cref{bernoulli} to give the desired result. 
\end{proof}

\begin{comment}
\begin{corollary}
$\BF_q(Y)$ has $\mathbb{Z}_q$-rank $(p-3)/2$ for any $q$.
\end{corollary}
\begin{proof}
This follows from the above theorem and the fact that 
\[\textup{rank}_{\mathbb{Z}_q}(M)=\textup{rank}_{\mathcal{O}}(M\otimes \mathcal{O})\] for any $\mathbb{Z}_q$-module $M$. 
\end{proof}
\DV{Nice!}
\end{comment}

As an equality in $\mathcal{O}[\Delta]$, we have
\begin{equation} \label{plus_minus}
e^{+} = \sum_{\substack{\psi \in \widehat{\Delta}(\mathcal{O}) \\ \psi \text{ even}}}e_{\psi} \text{ and } e^{-} = \sum_{\substack{\psi \in \widehat{\Delta}(\mathcal{O}) \\ \psi \text{ odd}}}e_{\psi},
\end{equation}
from which it follows that if $M^{-}$ is a finite $\mathcal{O}[\Delta]$-module (and by a finite $R$-module, we mean an $R$-module with finite cardinality, not a finitely generated $R$-module as it is sometimes meant), then 
\begin{equation} \label{deco_as_prod}
\# M^{-} = \prod_{\substack{\psi \in \widehat{\Delta}(\mathcal{O}) \\ \psi \text{ odd }}} \# e_{\psi} M,
\end{equation}
and a similar equality happens as well on the plus side.  \cref{key_iso1} allows us to calculate the cardinality of the odd isotypic components of ${\rm BF}_{\mathcal{O}}(Y)$.

\begin{corollary} \label{graph_side}
Let $Y/X$ be the Galois cover $Y/X$ defined in \cref{a_special_cover}, and let $\ell$ be a rational prime satisfying (\ref{ss_hyp}).  If $\ell \neq p$, and $\psi \in \widehat{\Delta}(\mathbb{C}_{\ell}) \subseteq \widehat{\Delta}(\mathcal{O})$ is an odd character, then
$$\# e_{\psi} {\rm BF}_{\mathcal{O}}(Y) = |\psi(p\theta)|_{\ell}^{-f_{\ell}} = |B_{1,\psi^{-1}}|_{\ell}^{-f_{\ell}},$$
where $f_{\ell} = [\mathcal{O}:\mathbb{Z}_{\ell}]$, and $|\cdot|_{\ell}$ denotes the usual $\ell$-adic absolute value on $\mathbb{C}_{\ell}$.  If $\ell=p$ and $\psi \in \widehat{\Delta}(\mathbb{C}_{p}) \subseteq \widehat{\Delta}(\mathbb{Z}_{p})$ is an odd character, then
$$\# e_{\psi} {\rm BF}_{p}(Y) = |\psi(p\theta)|_{p}^{-1} = p |B_{1,\psi^{-1}}|_{p}^{-1},$$
and for the trivial character $\psi_{0} \in \widehat{\Delta}(\mathbb{C}_{p}) \subseteq \widehat{\Delta}(\mathbb{Z}_{p})$, we have instead
$$\# e_{\psi_{0}} {\rm BF}_{p}(Y) = |\psi_{0}(p\theta)|_{p}^{-1} = p.$$
\end{corollary}
\begin{proof}
This follows directly from \cref{key_iso1} after noticing that when $\ell=p$ we have ${\rm BF}_{\mathcal{O}}(Y) = {\rm BF}_{p}(Y)$ and using again the equality
$$\psi_{0}(p\theta) = - \frac{(p-1)p}{2}.$$
\end{proof}
We recall the following result due to Mazur and Wiles from \cite[Theorem 2, page 216]{Mazur/Wiles:1984}.  
\begin{theorem}[Mazur-Wiles] \label{MW}
Let $K = \mathbb{Q}(\zeta_{p})$, and let $\ell$ be a rational prime satisfying (\ref{ss_hyp}).  With the notation as above, if $\ell \neq p$ and $\psi \in \widehat{\Delta}(\mathbb{C}_{\ell}) \subseteq \widehat{\Delta}(\mathcal{O})$ is any odd character, one has 
\begin{equation} \label{q_neq_p}
\# e_{\psi} {\rm Cl}^{-}_{\mathcal{O}}(K) = |B_{1,\psi^{-1}}|_{\ell}^{-f_{\ell}},
\end{equation}
where $f_{\ell} = [\mathcal{O}:\mathbb{Z}_{\ell}]$, and $|\cdot|_{\ell}$ denotes the usual $\ell$-adic absolute value on $\mathbb{C}_{\ell}$.  If $\ell = p$, then (\ref{q_neq_p}) is also true provided the odd character $\psi \in \widehat{\Delta}(\mathbb{C}_{p}) \subseteq \widehat{\Delta}(\mathbb{Z}_{p})$ satisfies $\psi \neq \omega$, where $\omega \in \widehat{\Delta}(\mathbb{Z}_{p})$ is the Teichm\"{u}ller character.
\end{theorem}

Combining with \cref{graph_side} gives the following corollary.

\begin{corollary} \label{equ_card}
With the notation as above, let $\ell$ be a rational prime satisfying (\ref{ss_hyp}).  If $\ell \neq p$ and $\psi \in \widehat{\Delta}(\mathbb{C}_{\ell}) \subseteq \widehat{\Delta}(\mathcal{O})$ is an odd character, then
$$\#e_{\psi}{\rm BF}_{\mathcal{O}}(Y) = \#e_{\psi}{\rm Cl}^{-}_{\mathcal{O}}(K).$$
If $\ell = p$ and $\psi \in \widehat{\Delta}(\mathbb{C}_{p}) \subseteq \widehat{\Delta}(\mathbb{Z}_{p})$ is an odd character different from the Teichm\"{u}ller character $\omega \in \widehat{\Delta}(\mathbb{Z}_{p})$, one has
$$\#e_{\psi} {\rm BF}_{p}(Y) = p \cdot \# e_{\psi} {\rm Cl}^{-}_{p}(K).$$
At last, we have
$$\#e_{\omega} {\rm BF}_{p}(Y) = \#e_{\omega} {\rm Cl}^{-}_{p}(K) = 1.$$
\end{corollary}
\begin{proof}
Let us assume first that $\ell \neq p$.  \cref{MW} gives for an odd character $\psi \in \widehat{\Delta}(\mathbb{C}_{\ell}) \subseteq \widehat{\Delta}(\mathcal{O})$ the equality
$$\#e_{\psi}{\rm Cl}^{-}_{\mathcal{O}}(K) = |B_{1,\psi^{-1}}|^{-f_{\ell}}.$$
Combining with \cref{graph_side} gives the first claim.  The same argument also gives the second claim.  For the Teichm\"{u}ller character $\omega$, \cite[Proposition 6.16]{Washington:1997} implies that $e_{\omega} {\rm Cl}^{-}_{p}(K) =0$, and since $|pB_{1,\omega^{-1}}|_{p}^{-1} = 1$ by \cref{bernoulli}, we have from \cref{graph_side} the chain of equalities
$$\# e_{\omega} {\rm BF}_{p}(Y) = |pB_{1,\omega^{-1}}|_{p}^{-1} = 1 = \# e_{\omega} {\rm Cl}^{-}_{p}(K). $$
\end{proof}

%Remove the following comment eventually.
\begin{comment}
We note first that \cref{key_iso1} implies that we have an isomorphism
$${\rm BF}_{p}(Y) \simeq {\rm BF}_{p}(Y)_{\rm tors} \times \mathbb{Z}_{p}^{r},$$
where $r$ is the number of even characters minus one, namely
$$r = \frac{p-3}{2}.$$
Moreover, one has
\begin{equation} \label{torsion_part}
{\rm BF}_{p}(Y)_{\rm tors} = {\rm BF}_{p}^{-}(Y) \oplus e_{\psi_{0}} {\rm BF}_{p}(Y).
\end{equation}
Combining (\ref{deco_as_prod}), (\ref{nf_side}), and \cref{equ_card}, we calculate
\begin{equation*}
\begin{aligned}
\# {\rm BF}_{p}^{-}(Y) &= \prod_{\psi \in \widehat{\Delta}_{{\rm odd}}(\mathbb{Z}_{p})} \# e_{\psi} {\rm BF}_{p}(Y) \\
&= \#e_{\omega}  {\rm BF}_{p}(Y) \cdot \prod_{\substack{\psi \in \widehat{\Delta}_{{\rm odd}}(\mathbb{Z}_{p}) \\ \psi \neq \omega}} \# e_{\psi} {\rm BF}_{p}(Y) \\
&= \#e_{\omega}  {\rm Cl}^{-}(K)[p^{\infty}] \cdot \prod_{\substack{\psi \in \widehat{\Delta}_{{\rm odd}}(\mathbb{Z}_{p}) \\ \psi \neq \omega}} \left(p \cdot \# e_{\psi} {\rm Cl}^{-}(K)[p^{\infty}]\right) \\
&= p^{\frac{p-3}{2}}  \cdot h_{p}^{-}(K),
\end{aligned}
\end{equation*}
and this gives the first equality in the statement of the corollary.  For the second equality, we know that $\#e_{\psi_{0}}{\rm BF}_{p}(Y) = p$ by \cref{equ_card}, and thus it suffices to combine the first equality obtained above with the direct sum decomposition (\ref{torsion_part}).  \DV{I don't know if I like this way of indicating a character is odd...}
\end{comment}

\appendix

\crefalias{section}{appendix}
\section{Galois theory for digraphs} \label{gr_acting_on_digraph}
In this appendix, we state and prove some fundamental results from the Galois theory of digraphs.
\begin{theorem}[Unique lifting property] \label{unique_lifting_prop}
Let $X,Y,Z$ be digraphs, and assume that $Z$ is strongly connected.  Suppose that $f:Y \rightarrow X$ is a covering map, and that $q,s:Z \rightarrow Y$ are two morphisms of digraphs satisfying $f \circ q = f \circ s$.  If there exists $w \in V_{Z}$ such that $q(w) = s(w)$, then $q = s$.
\end{theorem}
\begin{proof}
Let $e \in E_{Z,w}^{o}$.  We will show first that $q(e) = s(e)$.  Let $j = f \circ q = f \circ s$ and let $\varepsilon = j(e)$.   Note that
$$o(q(e)) = q(o(e)) = q(w) = s(w) = s(o(e)) = o(s(e)),$$
from which it follows that if we let $w_{0} = o(q(e)) = o(s(e))$, then $q(e), s(e) \in E_{Y,w_{0}}^{o}$.  Now,
$$f(q(e)) = \varepsilon = f(s(e))$$
and thus $q(e) = s(e)$ by the local bijectivity condition of $f$.  By the same token, if $e \in E_{Z,w}^{t}$, then $q(e) = s(e)$ as well.  Since $Z$ is assumed to be strongly connected, one gets the desired result.  
\end{proof}

If $Y/X$ is a cover of digraphs (see \cref{def_cov_map}), then the group of deck transformations ${\rm Aut}(Y/X)$ (see (\ref{deck_group_def})) acts on the fibers $f^{-1}(v)$ for all $v \in V_{X}$.  
\begin{definition}
We will call a cover $f:Y \rightarrow X$ of digraphs a Galois cover if both $X$ and $Y$ are strongly connected and if the action of ${\rm Aut}(Y/X)$ on each of the fibers $f^{-1}(v)$ is transitive for all $v \in V_{X}$.  When $Y/X$ is a Galois cover, we write ${\rm Gal}(Y/X)$ instead of ${\rm Aut}(Y/X)$.
\end{definition}

If $f:Y\rightarrow X$ is a cover of digraphs and $G = {\rm Aut}(Y/X)$, then we have a well-defined induced map $\tilde{f}_{G}:Y_{G} \rightarrow X$ given by
$$G \cdot w \mapsto \tilde{f}_{G}(G \cdot w) = f(w) \text{ and } G \cdot \varepsilon \mapsto \tilde{f}_{G}(G \cdot \varepsilon) = f(\varepsilon),$$
and we leave it to the reader to check that $\tilde{f}_{G}:Y_{G} \rightarrow X$ is a morphism of digraphs.
\begin{proposition} \label{galois_reformulation}
Let $X,Y$ be two strongly connected digraphs and let $f:Y \rightarrow X$ be a covering map.  Let also $G = {\rm Aut}(Y/X)$, then the cover $Y/X$ is Galois if and only if the morphism of digraphs
$$\tilde{f}_{G}: Y_{G} \rightarrow X $$
is an isomorphism.
\end{proposition}
\begin{proof}
Assume first that $\tilde{f}_{G}$ is an isomorphism of digraphs, and let $w_{1},w_{2} \in f^{-1}(v)$ for some $v \in V_{X}$.  Then, one has $\tilde{f}_{G}(G \cdot w_{1}) = \tilde{f}_{G}(G \cdot w_{2})$, and the assumption on $\tilde{f}_{G}$ implies that $G \cdot w_{1} = G \cdot w_{2}$.  In other words, $G$ acts transitively on $f^{-1}(v)$.

Conversely, assume that $G$ acts transitively on each of the fibers $f^{-1}(v)$.  We note first that since $f$ is surjective on vertices and edges, then so is $\tilde{f}_{G}$.  Therefore, to prove the claim, it suffices to show that $\tilde{f}_{G}$ is injective on vertices and edges.  If $\tilde{f}_{G}(G \cdot w_{1}) = \tilde{f}_{G}(G \cdot w_{2})$, then $f(w_{1}) = f(w_{2})$, and therefore the hypothesis shows the existence of $\sigma \in G$ such that $\sigma \cdot w_{1} = w_{2}$.  It follows that $G \cdot w_{1} = G \cdot w_{2}$, and thus $\tilde{f}_{G}$ is injective on vertices.  Similarly, if $\tilde{f}_{G}(G \cdot \varepsilon_{1}) = \tilde{f}_{G}(G \cdot \varepsilon_{2})$, then $f(\varepsilon_{1}) = f(\varepsilon_{2})$, from which it follows that $f(o(\varepsilon_{1})) = f(o(\varepsilon_{2}))$.  The hypothesis implies once more the existence of $\sigma \in G$ such that $\sigma \cdot o(\varepsilon_{1}) = o(\varepsilon_{2})$.  Set $w_{0} = \sigma \cdot o(\varepsilon_{1}) = o(\varepsilon_{2}) \in V_{Y}$.  Then $f(\sigma \cdot \varepsilon_{1}) = f(\varepsilon_{1}) = f(\varepsilon_{2})$.  Now, we have $\sigma \cdot \varepsilon_{1},\varepsilon_{2} \in E_{Y,w_{0}}^{o}$ and $f(\sigma \cdot \varepsilon_{1}) = f(\varepsilon_{2})$.  The local bijectivity condition satisfied by the cover $f$ gives $\sigma \cdot \varepsilon_{1} = \varepsilon_{2}$, and thus $G \cdot \varepsilon_{1} = G \cdot \varepsilon_{2}$ as desired.  

\end{proof}

There are two ways of constructing Galois covers of digraphs: a top-down approach and a bottom-up approach.  We explain here the former approach and the latter approach is described in \cref{volt_assignment}.  Start with a strongly connected digraph $Y$ and assume that a group $G$ acts on $Y$ freely on $V_{Y}$.  Note that the group morphism $\rho:G \rightarrow {\rm Aut}(Y)$ is necessarily injective, since the stabilizer of each vertex is trivial.  In this case, we can identify $G$ with its image in ${\rm Aut}(Y)$, and there is no loss of generality in assuming that $G$ is a subgroup of ${\rm Aut}(Y)$.  

\begin{theorem} \label{one_way_corr}
Let $Y$ be a strongly connected digraph and let $G$ be a subgroup of ${\rm Aut}(Y)$.  Assume also that $G$ acts freely on the vertices of $Y$.  For simplicity, let $X = Y_{G}$.  Then, the natural morphism $Y \rightarrow X$ is a Galois cover of digraphs for which
$${\rm Aut}(Y/X) = G.$$
\end{theorem}
\begin{proof}
Since the action of $G$ on $V_{Y}$ is free, the natural morphism $\pi_{G}:Y \rightarrow Y_{G} = X$ is a covering map by \cref{quotient_const}.  

The inclusion $G \subseteq {\rm Aut}(Y/X)$ is clear by definition of $\pi_{G}$.  On the other hand, if $\sigma \in {\rm Aut}(Y/X)$ and $w \in V_{Y}$, then $\pi_{G}(\sigma \cdot w) = \pi_{G}(w)$, and thus there exists $g \in G$ such that $\sigma \cdot w = g \cdot w$.  \cref{unique_lifting_prop} implies the equality $\sigma = g$, and thus $\sigma \in G$ as desired.

Since the action of $G$ is clearly transitive on the fibers $\pi_{G}^{-1}(G \cdot w)$, we get that $Y/X$ is a Galois cover, and this ends the proof.
\end{proof}

Our goal now for the rest of this section is to formulate and prove the Galois correspondence for digraphs (see \cref{gal_cor} below).

\begin{definition}
An intermediate cover of a cover $f:Y \rightarrow X$ consists of a triple $(Z,q,s)$, where $q:Y \rightarrow Z$ and $s:Z \rightarrow X$ are both covering maps satisfying $s \circ q = f$.  Two intermediate covers $(Z_{i},q_{i},s_{i})$, for $i=1,2$, are called equivalent if there exists an isomorphism of digraphs $\varphi: Z_{1} \rightarrow Z_{2}$ such that $\varphi \circ q_{1} = q_{2}$.
\end{definition}

We now associate to any subgroup $H$ of the group of deck transformations ${\rm Aut}(Y/X)$ an intermediate cover.

\begin{theorem}
Let $Y$ be a strongly connected digraph, $f:Y \rightarrow X$ a covering map, and $H$ a subgroup of ${\rm Aut}(Y/X)$.  Then, both morphisms of digraphs $\pi_{H}:Y \rightarrow Y_{H}$ and $\tilde{f}_{H}: Y_{H} \rightarrow X$ are covering maps.  In other words, the triple $(Y_{H},\pi_{H},\tilde{f}_{H})$ is an intermediate cover of $Y/X$.
\end{theorem}
\begin{proof}
Since $H \le {\rm Aut}(Y/X)$, the action of $H$ on $Y$ is free on $V_{Y}$ by \cref{act_special}.  It follows, as we already pointed out, that $Y \rightarrow Y_{H}$ is a cover.  Moreover, the induced map $\tilde{f}_{H}: Y_{H} \rightarrow X$ is a well-defined morphism of digraphs.  Since the equality $\tilde{f}_{H} \circ \pi_{H} = f$ is clear, it only remains to show that $\tilde{f}_{H}:Y_{H} \rightarrow Y$ is a covering map.  

The morphism $\tilde{f}_{H}:Y_{H} \rightarrow X$ is clearly surjective on vertices, so we just have to show the local bijectivity condition.  We show first that $\tilde{f}_{H}$ induces an injective function $E_{Y,H \cdot w}^{o} \rightarrow E_{X,f(w)}^{o}$.  If $H \cdot \varepsilon_{1}, H \cdot \varepsilon_{2} \in E_{Y_{H},H \cdot w}^{o}$, and $\tilde{f}_{H}(H \cdot \varepsilon_{1}) = \tilde{f}_{H}(H \cdot \varepsilon_{2})$, then $f(\varepsilon_{1}) = f(\varepsilon_{2})$.  Since $H \cdot o(\varepsilon_{1}) = H \cdot o(\varepsilon_{2}) = H \cdot w$, we have that $o(h_{1} \cdot \varepsilon_{1}) = o(h_{2} \cdot \varepsilon_{2}) = w$ for some $h_{1},h_{2} \in H$.  It follows that $h \cdot \varepsilon_{1}, h \cdot \varepsilon_{2} \in E_{Y,w}^{o}$ and 
$$f(h_{1} \cdot \varepsilon_{1}) = f(\varepsilon_{1}) = f(\varepsilon_{2}) = f(h_{2} \cdot \varepsilon_{2});$$
thus $h \cdot \varepsilon_{1} = h \cdot \varepsilon_{2}$, since $f$ itself is a cover.  We deduce that $H \cdot \varepsilon_{1} = H \cdot \varepsilon_{2}$, and that $\tilde{f}_{H}$ induces an injective function $E_{Y,H \cdot w}^{o} \rightarrow E_{X,f(w)}^{o}$.  The rest of the argument is similar and left to the reader.
\end{proof}

Starting with an intermediate cover, one obtains a subgoup as follows.
\begin{theorem}
Let $f:Y \rightarrow X$ be a cover, and let $(Z,q,s)$ be an intermediate cover.  Then, ${\rm Aut}_{q}(Y/Z) \le {\rm Aut}_{f}(Y/X)$.  Moreover, if $(Z_{1},q_{1},s_{1})$ is an intermediate cover equivalent to another one, say $(Z_{2},q_{2},s_{2})$ via an isomorphism of digraphs $\varphi:Z_{1} \rightarrow Z_{2}$, then ${\rm Aut}_{q_{1}}(Y/Z_{1}) = {\rm Aut}_{q_{2}}(Y/Z_{2})$.
\end{theorem}
\begin{proof}
Let $\sigma \in {\rm Aut}_{q}(Y/Z)$, then we have $q \circ \sigma = q$ by definition.  Thus, one has
$$f \circ \sigma = s \circ q \circ \sigma = s \circ q = f.$$
It follows that $\sigma \in {\rm Aut}_{f}(Y/X)$ showing the inclusion ${\rm Aut}_{q}(Y/Z) \subseteq {\rm Aut}_{f}(Y/X)$.

For the second part or the theorem, if $\sigma \in {\rm Aut}_{q_{1}}(Y/Z_{1})$, then $q_{1} \circ \sigma = q_{1}$.  But then
$$q_{2} \circ \sigma = \varphi \circ q_{1} \circ \sigma = \varphi \circ q_{1} = q_{2},$$
which shows the inclusion ${\rm Aut}_{q_{1}}(Y/Z_{1}) \subseteq {\rm Aut}_{q_{2}}(Y/Z_{2})$.  The other inclusion is similar and left to the reader.
\end{proof}

The following result is important for the Galois correspondence and also to obtain other Galois covers from intermediate covers.

\begin{theorem} \label{galois_top}
Let $f:Y \rightarrow X$ be a Galois cover of digraphs and let $(Z,q,s)$ be an intermediate cover.  Then, the cover $q:Y \rightarrow Z$ is also a Galois cover.
\end{theorem}
\begin{proof}
We have to show that ${\rm Aut}(Y/X)$ acts transitively on the fibers of $q$.  Let $w_{1},w_{2} \in V_{Y}$ be such that $q(w_{1}) = q(w_{2})$.  Then,
$$f(w_{1}) = s \circ q(w_{1}) = s \circ q (w_{2}) = f(w_{2}).$$
Since $f:Y \rightarrow X$ is Galois, there exists $\sigma \in {\rm Gal}(Y/X)$ such that $\sigma \cdot w_{1} = w_{2}$.  Applying \cref{unique_lifting_prop} to the diagram
\begin{equation*}
\begin{tikzcd}
Y \arrow[r, "q \circ \sigma", yshift=0.5ex] \arrow[r,"q" ',shift right=0.5ex]  & \arrow[d,"s"]Z \\
& X
\end{tikzcd}
\end{equation*}
gives that $\sigma \in {\rm Aut}(Y/Z)$, and this ends the proof.
\end{proof}

At last, we can formulate and prove a Galois correspondence for Galois covers of digraphs.
\begin{theorem}[Galois correspondence] \label{gal_cor}
Let $f:Y \rightarrow X$ be a Galois cover of digraphs.  Then, there is a one-to-one correspondence between subgroups of ${\rm Gal}(Y/X)$ and equivalence classes of intermediate covers of $Y/X$.  This correspondence is obtained as follows.  If $H \le {\rm Gal}(Y/X)$, then the corresponding equivalence class of intermediate covers is the equivalence class of $(Y_{H},\pi_{H},\tilde{f}_{H})$.  Conversely, if $(Z,q,s)$ is a representative for an equivalence class of intermediate covers, then the corresponding subgroup is ${\rm Gal}(Y/Z)$.
\end{theorem}
\begin{proof}
The first half of the correspondence has nothing to do with Galois covers.  If $Y$ is a strongly connected digraph, and $f:Y \rightarrow X$ is a cover, not necessarily Galois, then starting with a subgroup $H \le {\rm Aut}(Y/X)$, if follows from \cref{one_way_corr} that ${\rm Aut}(Y/Z) = H$, where we set $Z = Y_{H}$ for ease of notation.  This shows that the correspondence that associates a subgroup $H$ to the intermediate cover $(Y_{H},\pi_{H},\tilde{f}_{H})$ is injective, and that the one that associates to an equivalence class of intermediate covers a corresponding subgroup is surjective.

The second half of the correspondence is valid only under the Galois assumption and follows from \cref{galois_top} as follows.  Let $(Z,q,s)$ be an intermediate cover and set $H = {\rm Aut}(Y/Z)$.  We have to show that $(Z,q,s)$ is equivalent to $(Y_{H},\pi_{H},\tilde{f}_{H})$.  Consider the natural morphism of digraphs $\tilde{q}_{H}:Y_{H} \rightarrow Z$.  Since $Y/Z$ is Galois by \cref{galois_top}, it follows from \cref{galois_reformulation} that the morphism $\tilde{q}_{H}$ is an isomorphism of digraphs.  Moreover, a simple calculation shows that
$$\tilde{q}_{H} \circ \pi_{H} = q \text{ and } s \circ \tilde{q}_{H} = \tilde{f}_{H}.$$
This shows the desired claim, and ends the proof of the Galois correspondence.
\end{proof}

We end this section with the following useful result.

\begin{theorem} \label{last_gal_result}
Let $f:Y \rightarrow X$ be a Galois cover of digraphs, and let $(Z,q,s)$ be an intermediate cover whose equivalence class corresponds to the subgroup $H = {\rm Gal}(Y/Z)$.  Then $s:Z \rightarrow X$ is a Galois cover if and only if $H$ is a normal subgroup of ${\rm Gal}(Y/X)$ in which case we have a natural group isomorphism
$${\rm Gal}(Y/X)/{\rm Gal}(Y/Z) \stackrel{\simeq}{\longrightarrow}{\rm Gal}(Z/X).$$
\end{theorem}
\begin{proof}
Assume that the cover $s:Z \rightarrow X$ is Galois, and let us show first that there is a natural group morphism
$${\rm res}:{\rm Gal}(Y/X) \rightarrow {\rm Gal}(Z/X).$$
We claim that given $\sigma \in {\rm Gal}(Y/X)$, there exists a unique $\gamma \in {\rm Gal}(Z/X)$ such that the diagram
\begin{equation} \label{com}
\begin{tikzcd}
Y \arrow[r, "\sigma"] \arrow[d,"q" ']  & \arrow[d,"q"] Y \\
Z \arrow[r,"\gamma"] &  Z
\end{tikzcd}
\end{equation}
commutes.  The uniqueness part follows immediately from the fact that $q:Y \rightarrow Z$ is a covering maps, and hence surjective on both vertices and edges.  To show the existence, let $w \in V_{Y}$.  Then, $q(w), q(\sigma \cdot w) \in V_{Z}$ and note that
$$s(q(w)) = f(w) = f(\sigma \cdot w) = s(q(\sigma \cdot w)).$$
Since the cover $s:Z \rightarrow X$ is assumed to be Galois, there exists $\gamma \in {\rm Gal}(Z/X)$ such that
$$\gamma \circ q(w) = q \circ \sigma(w).$$
Applying \cref{unique_lifting_prop} to the diagram
\begin{equation} \label{useful_twice}
\begin{tikzcd}
Y \arrow[r, "\gamma \circ q", yshift=0.5ex] \arrow[r,"q \circ \sigma" ',shift right=0.5ex]  & \arrow[d,"s"]Z \\
& X
\end{tikzcd}
\end{equation}
gives the commutativity of the diagram (\ref{com}).  We can now define ${\rm res}:{\rm Gal}(Y/X) \rightarrow {\rm Gal}(Z/X)$ via $\sigma \mapsto {\rm res}(\sigma) = \gamma$, where $\gamma$ is the unique element in ${\rm Gal}(Z/X)$ such that $q \circ \sigma = \gamma \circ q$ whose existence we just proved.  The fact that the map ${\rm res}$ is a group morphism is left to the reader.  We claim that ${\rm res}$ is surjective.  If $\gamma \in {\rm Gal}(Z/X)$, then let $w_{1} \in V_{Y}$ and consider $\gamma \circ q (w_{1}) \in V_{Z}$.  Since $q$ is surjective on vertices, there exists $w_{2} \in V_{Y}$ such that $q(w_{2}) = \gamma \circ q(w_{1})$.  It follows that 
$$f(w_{1}) = s \circ q(w_{1}) = s(\gamma \circ q(w_{1})) = s(q(w_{2})) = f(w_{2}),$$
and since $f:Y\rightarrow X$ is assumed to be Galois, there exists $\sigma \in {\rm Gal}(Y/X)$ such that $\sigma \cdot w_{1} = w_{2}$.  Applying again \cref{unique_lifting_prop} to the diagram (\ref{useful_twice}) gives $q \circ \sigma = \gamma \circ q$, and this shows the surjectivity of the group morphism ${\rm res}$.  The fact that ${\rm ker}({\rm res}) = {\rm Gal}(Y/Z)$ is simple and left to the reader.  We then obtain that ${\rm Gal}(Y/Z)$ is a normal subgroup of ${\rm Gal}(Y/X)$ and this ends one direction of the proof.

Conversely, let us assume that $H = {\rm Gal}(Y/Z) \unlhd {\rm Gal}(Y/X)$.  The normality assumption implies that given any $\sigma \in {\rm Gal}(Y/X)$, the map $\sigma_{H}:Y_{H} \rightarrow Y_{H}$ given by
$$\sigma_{H}(H \cdot w) = H \cdot \sigma(w) \text{ and } \sigma_{H}(H \cdot e) = H \cdot \sigma(e)$$
is a well-defined isomorphism of digraphs.  Moreover, since $Y/Z$ is Galois by \cref{galois_top}, we have by \cref{galois_reformulation} that the induced map $\tilde{q}_{H}:Y_{H} \rightarrow Z$ is an isomorphism of digraphs.  Consider now $\tilde{q}_{H} \circ \sigma_{H} \circ \tilde{q}_{H}^{-1} \in {\rm Aut}(Z)$.  The situation is summarized in the following commutative diagram
\begin{equation*}
\begin{tikzcd}
& \arrow[dl,"q"'] \arrow[d,"\pi_{H}"]Y \arrow[r,"\sigma"] &Y \arrow[d,"\pi_{H}" '] \arrow[dr,"q"]& \\
Z & \arrow[l,"\tilde{q}_{H}"]  Y_{H} \arrow[r,"\sigma_{H}"'] &  Y_{H} \arrow[r,"\tilde{q}_{H}"'] & Z
\end{tikzcd}
\end{equation*}
where all horizontal arrows are isomorphisms of digraphs.  We leave it to the reader to check that $\tilde{q}_{H} \circ \sigma_{H} \circ \tilde{q}_{H}^{-1} \in {\rm Aut}(Z/X)$.  We can now show that $Z/X$ is Galois.  Let $v_{1},v_{2} \in V_{Z}$ be such that $s(v_{1}) = s(v_{2})$, and let $w_{1},w_{2} \in V_{Y}$ be such that $q(w_{i})=v_{i}$ for $i=1,2$.  Then $f(w_{1}) = f(w_{2})$, and since $Y/X$ is Galois, there exists $\sigma \in {\rm Gal}(Y/X)$ such that $\sigma \cdot w_{1} = w_{2}$.  We now observe that
\begin{equation*}
\begin{aligned}
\tilde{q}_{H} \circ \sigma_{H} \circ \tilde{q}_{H}^{-1}(v_{1}) &= \tilde{q}_{H} \circ \sigma_{H}(H \cdot w_{1}) \\
&= \tilde{q}_{H}(H \cdot \sigma(w_{1})) \\
&= q(\sigma(w_{1})) \\
&= q(w_{2}) \\
&= v_{2}.
\end{aligned}
\end{equation*}
It follows that ${\rm Aut}(Z/X)$ acts transitively on the fibers of $s$, and thus $Z/X$ is a Galois cover, as we wanted to show.
\end{proof}

\begin{remark} \label{undirected_galois}
Recall that a (undirected) graph in the sense of Serre is a digraph $X = (V_{X},\mathbf{E}_{X})$ equipped with an inversion map ${\rm inv}:\mathbf{E}_{X} \rightarrow \mathbf{E}_{X}$ denoted by $e \mapsto {\rm inv}(e) = \bar{e}$ satisfying:
\begin{enumerate}
\item $\bar{e} \neq e$ for all $e \in \mathbf{E}_{X}$,
\item $\bar{\bar{e}} = e$ for all $e \in \mathbf{E}_{X}$,
\item $t(\bar{e}) = o(e)$ and $o(\bar{e}) = t(e)$ for all $e \in \mathbf{E}_{X}$.
\end{enumerate}
Let $G$ be a group acting on a graph $Y$.  In order for the quotient $G \backslash Y = (G \backslash V_{Y},G \backslash \mathbf{E}_{Y})$ to be a graph with the inversion map induced from the one on $Y$, one needs the action to be without inversion, meaning that $\sigma \cdot \varepsilon \neq \bar{\varepsilon}$ for all $\varepsilon \in \mathbf{E}_{Y}$.  All of \cref{gr_acting_on_digraph} works also for graphs as explained in \cite[Sections 5.1 and 5.2]{Sunada:2013}, in other words, one can view the Galois theory for graphs as the one for digraphs with the extra structure of the inversion map taken into account and assuming throughout that groups act without inversion on graphs.  
\end{remark}

\begin{remark}
All of \cref{volt_assignment} can also be modified accordingly in the category of (undirected) graphs.  In order to obtain a graph, the function $\alpha:\mathbf{E}_{X} \rightarrow G$ has to satisfy the extra condition $\alpha(\bar{e}) = \alpha(e)^{-1}$.  The inversion map is then defined to be
$$ (e,\sigma) \mapsto \overline{(e,\sigma)} = (\bar{e},\sigma \alpha(e)),$$
and one can check that the three properties of the inversion map listed in \cref{undirected_galois} are satisfied.
\end{remark}

%\section{An alternative proof of Theorem~\ref{thm:m(Y)}}
\section{Bowen--Franks operator as a circulant matrix.}\label{app:alternative}
In this appendix, we present a different proof of \cref{thm:m(Y)}. Our approach relies on the observation that, in the specific context under consideration, the Bowen--Franks operator can be represented by a circulant matrix.  This approach can be extended to more general situations and may also be of independent interest.

\textbf{Step 1: Realizing \texorpdfstring{$\cBF$}{} as a linear operator given by the multiplication by a polynomial.}
Let $\tau=\sigma_{i}$ be a generator of the finite abelian group $\Delta$. In particular $i$ is a primitive modulo $p$. For an integer $k$ that is coprime to $p$, we write $[k]$ for the unique integer $1\le[k]\le p-1$ such that $[k]\equiv k\mod p$. The matrix of $\cBF$ with respect to the basis $\{1,\tau,\tau^2,\dots,\tau^{p-2}\}$ as a $\Z$-linear map on $\Z[\Delta]$ is given by
\[
M:=-\begin{bmatrix}
    1&[i]&\cdots &[i^{p-2}]\\
    [i^{-1}]&1&\cdots& [i^{p-3}]\\
        [i^{-2}]&[i^{-1}]&\cdots& [i^{p-4}]\\
    \vdots&\vdots&&\vdots\\
    [i^{-p+2}]&[i^{-p+3}]&\cdots&1
\end{bmatrix}=
-\begin{bmatrix}
    1&[i]&\cdots &[i^{p-2}]\\
    [i^{p-2}]&1&\cdots& [i^{p-3}]\\
    [i^{p-3}]&[i^{p-2}]&\cdots&[i^{p-4}]\\
    \vdots&\vdots&&\vdots\\
    [i]&[i^2]&\cdots&1
\end{bmatrix}
\]
%\KM{I think the last line should be $[i], [i^2],...$ because $\tau^{-p+2}=\tau^{-(p-1)-1+2}=\tau$}
In particular, this is a circulant matrix. We shall regard the $\ZZ$-linear map $\cBF$ on $\ZZ[\Delta]$ as the multiplication by $M$ on $\ZZ^{p-1}$.

Define the following $\ZZ$-linear isomorphism
\begin{align*}
\Phi:\ZZ[x]/(x^{p-1}-1)&\to\ZZ^{p-1}\\
x^i&\mapsto e_{i+1},
\end{align*}
where $e_1,e_2,\dots,e_{p-1}$ is the canonical basis of $\ZZ^{p-1}$.
Let $$f(x)=1+[i^{p-2}]x+[i^{p-3}]x^2+\cdots+[i]x^{p-2}$$ %\KM{Shouldn't it be $[i^{-1}]$ and so forth?}
and write $\tilde\cBF$ for the  multiplication by $-f(x)$ map on $\ZZ[x]/(x^{p-1}-1)$.  The matrix of $\Phi\circ\tilde\cBF\circ\Phi^{-1}$ with respect to the canonical basis coincides with $M$. Therefore, $\cBF$ and $\tilde\cBF$ have isomorphic images. Note that
\begin{equation}
\Im(\tilde\cBF)=\left(f(x),x^{p-1}-1\right)/(x^{p-1}-1).    
\label{eq:image-BF-tilde}
\end{equation}

\textbf{Step 2: Description of \texorpdfstring{$\cBF$}{} after extension of scalar by \texorpdfstring{$\CC$}{}.}
Since $M$ is a circulant matrix, it is diagonalizable over $\CC$ and its eigenvalues are given by $-f(\zeta^{j})$, where $\zeta=e^{2\pi i/(p-1)}$, $j=0,1,\dots p-2$. In addition, the vector $v_j:=(1,\zeta^j,\zeta^{2j},\dots,\zeta^{(p-2)j})\in\CC^{p-1}$ belongs to the eigenspace of $-f(\zeta^{j})$. Note that $v_1,\dots,v_{p-1}$ form a basis of $\CC^{p-1}$ since they are orthogonal to each other.

Indeed, if $\psi_j$ denotes the character of $\Delta$ that sends $\tau$ to $\zeta^j$, we recover the eigenvalues of $\cBF$:
\begin{equation}
-f(\zeta^j)=-\sum_{k=0}^{p-2} [i^{-k}]\zeta^{jk}=-\sum_{n=1}^{p-1} n\psi_j^{-1}(n)=\begin{cases}
    -\frac{p(p-1)}{2}&j=0,\\
       -pB_{1,\psi_j^{-1}}\ne0&j\in\{1,3,\dots,p-2\},\\
     0&j\in\{2,4,\dots, p-3\}.
\end{cases}
\label{eq:e-vals}    
\end{equation}
%\KM{If my remark above is right, substitute $\psi_j$ by $\psi_j^{-1}$}
In particular, the rank of $M$ is equal to $(p+1)/2$.

Let $S=\{0,1,3,\dots,p-2\}$ and $S'=\{2,4,\dots,p-3\}$. We have 
\begin{equation}
\Im(\cBF)\otimes\CC=\langle v_j:j\in S\rangle_\CC =\langle v_j: f(\zeta^j)\ne0\rangle_\CC    \label{eq:image}
\end{equation}
since $\{-f(\zeta^j):j\in S\}$ is the set of non-zero eigenvalues of $\cBF$. 
%Hence, we deduce that
%\[
%\Im(\cBF)\subseteq\ZZ^{p-1}\bigcap\langle v_j:j\in S\rangle_\CC .
%\]

\textbf{Step 3: Finding a direct summand that contains the image of \texorpdfstring{$\tilde{\cBF}$}{}.}
Suppose that $g(x)\in\Z[x]$ is a monic polynomial that divides $x^{p-1}-1$. There is a natural short exact sequence of $\Z$-modules
\[
0\to (g(x))/(x^{p-1}-1)\to\Z[x]/(x^{p-1}-1)\to \Z[x]/(g(x))\to0.
\]
As a $\ZZ$-submodule, $T_g:=(g(x))/(x^{p-1}-1)$ is a direct summand of $\Z[x]/(x^{p-1}-1)$ since $\left(\Z[x]/(x^{p-1}-1)\right)/T_g$ is a free $\ZZ$-module. Note that the rank of $T_g$ is $p-1-\deg(g)$.

Let $k(x)=(x^{p-1}-1)/g(x)$. Note that $T_g$ is the kernel of the multiplication by $k(x)$ map on $\Z[x]/(x^{p-1}-1)$. Similar to $\tilde\cBF$, the eigenvalues of this map are given by $k(\zeta^j)$, $j=0,1,\dots,p-2$. Hence, $T_g\otimes\CC$ is spanned by the eigenvectors for the $0$-eigenspace. More explicitly,
\[
\Phi(T_g)\otimes\CC=\langle v_j:k(\zeta^j)=0\rangle_\CC.
\]
Hence, combined with \eqref{eq:image}, we deduce that
\begin{equation}
\Im(\tilde\cBF)\subseteq T_g,        
\label{eq:inclusion}
\end{equation}
if $g(x)$ is set as $\prod_{j\in S'}(x-\zeta^j)$. Here, $T_g$ is a direct summand of rank $(p+1)/2$ inside $\ZZ[x]/(x^{p-1}-1)$. Furthermore, it follows from \eqref{eq:e-vals} that $g(x)$ equals $\gcd(f(x),x^{p-1}-1)$.

Set $h(x)=f(x)/g(x)$, and $k(x)=(x^{p-1}-1)/g(x)$. In particular, $\gcd(h(x),k(x))=1$.
 By \eqref{eq:image-BF-tilde} and \eqref{eq:inclusion},
we have
\[
\frac{T_g}{\Im(\tilde\cBF)}=\frac{(g(x))/(x^{p-1}-1)}{(f(x),x^{p-1}-1)/(x^{p-1}-1)}\cong\frac{(g(x))}{(f(x),x^{p-1}-1)}\cong\frac{\ZZ[x]}{(h(x),k(x))}.
\]
As $\gcd(h(x),k(x))=1$, this quotient is finite. Hence, $T_g$ is the saturation of $\Im(\tilde{\cBF})$ in $\Z[x]/(x^{p-1}-1)$.

The lattice $\mathcal{L}(X)$ in Definition~\ref{def:L(X)} is the saturation of $\Im(\cBF)$ in $\ZZ[\Delta]\cong\ZZ^{p-1}$.
Therefore, $\mathcal{L}(X)$ can be identified with $T_g$ via $\Phi$.

\textbf{Step 4: Calculating the index \texorpdfstring{$m(Y)$}{}.} Our discussion in Step 3, combined with \cref{m_as_card}, tells us that
\[
|m(Y)|=\left(\Im(\tilde\cBF):\tilde\cBF(T_g)\right),
\]
 which is equal to the cardinality of the quotient
\[
\frac{\left(f(x),x^{p-1}-1\right)}{\left(f(x)g(x),x^{p-1}-1\right)}.
\]
As all polynomials appearing in this expression are divisible by $g(x)$, this is isomorphic to
\[
\frac{\left(h(x),k(x)\right)}{\left(f(x),k(x)\right)}.
\]
Applying the second isomorphism gives
\[
\frac{\left(h(x)\right)}{\left(f(x),k(x)\right)\bigcap(h(x))}\cong\frac{\left(h(x)\right)}{\left(f(x),h(x)k(x)\right)},
\]
since $h(x)|f(x)$ and $\gcd(h(x),k(x))=1$. After factoring out $h(x)$, we see that this is isomorphic to
\[
\ZZ[x]/(g(x),k(x)),
\]
the cardinality of which is given by the absolute value of the resultant of $g(x)$ and $k(x)$.

Explicitly, we have
\begin{align*}
g(x)&=\prod_{j\in S'}(x-\zeta^j)=\frac{x^{(p-1)/2}-1}{x-1}=1+x+\cdots x^{\frac{p-3}2},\\
k(x)&=\prod_{j\in S}(x-\zeta^j)=(x-1)(x^{(p-1)/2}+1).
\end{align*}
Therefore, the index is given by
\[
\left|g(1)\prod_{j\in S\setminus\{0\}}g(\zeta^j)\right|.
\]
We can check that
\[
g(1)=\frac{p-1}{2},
\]
and
\[
\left|\prod_{j\in S\setminus\{0\}}g(\zeta^j)\right|=\left|\frac{\prod_{j\in S\setminus\{0\}}(\zeta^{(p-1)/2}-1)}{\prod_{j\in S\setminus\{0\}} (\zeta^j-1)}\right|=\left|\frac{\prod_{j\in S\setminus\{0\}}(-1-1)}{(1)^{(p-1)/2}+1}\right|=\frac{2^{(p-1)/2}}{2}=2^{(p-3)/2}
\]
since $x^{(p-1)/2}+1=\prod_{j\in S\setminus\{0\}}(x-\zeta^j)$. Therefore, we conclude that
\[
|m(Y)|=\frac{p-1}{2}\cdot 2^{(p-3)/2}.
\]

\bibliographystyle{alpha}
\bibliography{references}
\end{document}